\documentclass[11pt,reqno]{amsart}

\usepackage[utf8]{inputenc}
\usepackage[margin=1.25in]{geometry}
\usepackage{hyperref}
\usepackage{appendix}
\usepackage{amsfonts}
\usepackage{amsthm}
\usepackage{amssymb}
\usepackage{stmaryrd} 
\usepackage{amsmath}
\usepackage{amsthm}
\usepackage[dvipsnames]{xcolor}
\usepackage{mathrsfs}
\usepackage{lipsum} 

\theoremstyle{plain}
\newtheorem{theorem}{Theorem}[section]
\newtheorem{corollary}[theorem]{Corollary}
\newtheorem{lemma}[theorem]{Lemma}
\newtheorem{Proposition}[theorem]{Proposition}

\newtheorem{Definition}[theorem]{Definition}

\theoremstyle{remark}

\newtheorem{remark}[theorem]{Remark}
\newtheorem*{unnumberedremark}{Remark}

\usepackage{biblatex} 
\numberwithin{equation}{section}
\title[Smallest singular value Lévy entries]{The smallest singular value for rectangular random matrices with Lévy entries}

\author{Yi HAN}
\address{Department of Mathematics, Massachusetts Institute of Technology, Cambridge, MA
}
\email{hanyi16@mit.edu}

\begin{document}

\begin{abstract}
Let $X=(x_{ij})\in\mathbb{R}^{N\times n}$ be a rectangular random matrix with i.i.d. entries (we assume $N/n\to\mathbf{a}>1$), and denote by $\sigma_{min}(X)$ its smallest singular value.
When entries have mean zero and unit second moment, the celebrated work of Bai-Yin and Tikhomirov show that $n^{-\frac{1}{2}}\sigma_{min}(X)$ converges almost surely to $\sqrt{\mathbf{a}}-1.$ However, little is known when the second moment is infinite. In this work we consider symmetric entry distributions satisfying $\mathbb{P}(|x_{ij}|>t)\asymp t^{-\alpha}$ for some $\alpha\in(0,2)$, and
 prove that $\sigma_{min}(X)$ can be determined up to a log factor with high probability: for any $D>0$, with probability at least $1-n^{-D}$ we have 
$$C_1n^{\frac{1}{\alpha}}(\log n)^\frac{2(\alpha-2)}{\alpha}\leq \sigma_{min}(X)\leq C_2n^{\frac{1}{\alpha}}(\log n)^\frac{\alpha-2}{2\alpha}$$ for some constants $C_1,C_2>0$. The upper bound was derived in a recent work of Bao, Lee and Xu \cite{bao2024phase2} but the lower bound is new and answers a problem posed in that paper in a weaker form. This appears to be the first determination of $\sigma_{min}(X)$ in the $\alpha$-stable case with a correct leading order of $n$, as previous anti-concentration arguments only yield lower bound $n^\frac{1}{2}$. The same lower bound holds for $\sigma_{min}(X+B)$ for any fixed rectangular matrix $B$ with no assumption on its operator norm. 
The case of diverging aspect ratio is also computed. 

\end{abstract}

\maketitle

\section{Introduction}

This paper is concerned with the smallest singular value of a rectangular random matrix with i.i.d. entries. A classical theorem of Bai and Yin \cite{MR1235416} is as follows:
\begin{theorem}\label{theorem1.111}[\cite{MR1235416}] 
    Let $X=(x_{ij})\in\mathbb{R}^{N\times n}$ with independent, identically distributed (i.i.d.) entries $x_{ij}$ satisfying $\mathbb{E}[x_{ij}]=0$, $\mathbb{E}[|x_{ij}|^2]=1$, $\mathbb{E}[|x_{ij}|^4]<\infty$. Assume that the aspect ratio satisfies $N/n\to \mathbf{a}$ as $n\to\infty$ for some $\mathbf{a}>1$. Then almost surely 
    $$n^{-1/2} \sigma_{min}(X)\to \sqrt{\mathbf{a}}-1,\quad\text{ as } n\to\infty,
    $$ where $\sigma_{min}(X)$ denotes the smallest singular value of $X$.
\end{theorem}

Under the assumptions of Theorem \ref{theorem1.111},  $n^{-\frac{1}{2}}\sigma_{max}(X)$ also converges almost surely to $\sqrt{\mathbf{a}}+1$, where $\sigma_{max}(X)$ is the largest singular value of $X$. The empirical spectral measure of $N^{-1}X^*X$ converges to the Marchenko-Pastur law with aspect parameter $1/\mathbf a$ \cite{marchenko1967distribution} and when $x_{ij}$ are sub-Gaussian, Tracy-Widom fluctuations of $\sigma_{min}(X)$ and $\sigma_{max}(X)$ are proven in \cite{soshnikov2002note}, \cite{feldheim2010universality}. For regularly varying entry tails with exponent $\alpha\in(0,4)$, assuming centering when $\alpha\geq2$, suitably rescaled extreme singular values converge to a Poisson point process, while the rescaled largest singular value has a Fr\'echet limit \cite{auffinger2009poisson}. Recently, \cite{han2024deformed} determined the law of $\sigma_{max}(X)$ when $\mathbb{P}(|x_{ij}|\geq t)\sim t^{-4}$.

It remained an open problem for some time whether the finite fourth moment condition is necessary for a.s. convergence of $n^{-1/2}\sigma_{min}(X)$ to $\sqrt{\mathbf{a}}-1$, until Tikhomirov finally removed the fourth moment condition in \cite{tikhomirov2015limit}.

\begin{theorem}\label{ticks}[\cite{tikhomirov2015limit}] The claim of Theorem \ref{theorem1.111} is true without assuming $\mathbb{E}[|x_{ij}|^4]<\infty.$
\end{theorem}

Recently, \cite{bao2024phase} investigated the finer asymptotics of $\sigma_{min}(X)$ when $\mathbb{P}(|x_{ij}|\geq t)\sim t^{-\alpha}$, $\alpha\in(2,4)$ and obtained three different scaling limits in different ranges of $\alpha$.

Having a good understanding of $\sigma_{min}(X)$ when $x_{ij}$ has a finite second moment, a natural question is to study $\sigma_{min}(X)$ when $x_{ij}$ has heavier tail, say they have a regularly varying tail with $\mathbb{P}(|x_{ij}|\geq t)\sim t^{-\alpha}$ for some $\alpha\in(0,2)$. In this direction only partial progress has been achieved so far.

A lower bound for $\sigma_{min}(X)$ has been derived in \cite{tikhomirov2016smallest} without assuming any moment condition on $x_{ij}$, and is purely based on anti-concentration properties:

\begin{theorem}\label{theorem1.2345}[\cite{tikhomirov2016smallest}](Informal statement) For any $\delta>1$ and $N\geq \delta n$, consider a rectangular random matrix
 $X=(X_{ij})\in\mathbb{R}^{N\times n}$ with i.i.d. entries that are uniformly anti-concentrated (measured in terms of Lévy concentration function \eqref{levyconcentrationf}). Then we can find $u,v>0$ such that, for any non-random $B\in\mathbb{R}^{N\times n}$, we have 
$$
\mathbb{P}(\sigma_{min}(X+B)\leq u\sqrt{n})\leq \exp(-vn).
$$    
\end{theorem}

Applying Theorem \ref{theorem1.2345} to a rectangular matrix $X$ with $\alpha$-stable distribution, we deduce that with very high probability $\sigma_{min}(X)\geq C\sqrt{n}$ for some $C>0$. However, one may expect that this bound significantly underestimates the magnitude of $\sigma_{min}(X)$ when $\alpha\in(0,2)$, as we will prove that $\sigma_{min}(X)$ should be of order $n^{\frac{1}{\alpha}}$ modulo some log factor.

A closely related setting
is square Lévy matrices,
where the correct asymptotics of $\sigma_{min}(X)$ is recently derived in Louvaris \cite{louvaris2022universality}. Specifically, 

\begin{theorem}\label{theorem1.4gagag}[\cite{louvaris2022universality}] For $\alpha\in(0,2)$  Let $\sigma$ be defined by equation (1.2) in \cite{louvaris2022universality}. let $Z$ be a random variable following a $(0,\sigma)$ $\alpha$-stable law in the sense that 
$$
\mathbb{E}[e^{itZ}]=\exp(-\sigma^\alpha|t|^\alpha)\text{ for all } t\in\mathbb{R}.
$$ Let $D_n(\alpha)=(d_{ij})_{1\leq i,j\leq n}$ be a squared random matrix with i.i.d. entries, each entry has the law $n^{-\frac{1}{\alpha}}Z$. Then there is a countable set $\mathcal{A}\subset(0,2)$ such that for all $\alpha\in(0,2)\setminus\mathcal{A}$, we have
$$\mathbb{P}(n\xi
\sigma_{min}(D_n(\alpha))\leq r)=1-\exp(-\frac{r^2}{2}-r)+\mathcal{O}_r(n^{-c})
$$ for some $c>0$, where $\xi$ is a constant depending only on $\alpha$ defined in \cite{louvaris2022universality}, eq. (1.4). Here $\mathcal{O}_r(n^{-c})$ denotes a quantity bounded in absolute value by $C(r)n^{-c}$ where $C(r)>0$ is a constant depending only on $r>0$.
\end{theorem}
Theorem \ref{theorem1.4gagag} generalizes the celebrated universality result of Tao and Vu \cite{tao2010random} for the least singular value of i.i.d. (light-tailed) matrices to i.i.d. Lévy matrices.

Now we turn back to a rectangular matrix with i.i.d. stable distribution. In this regime, the spectral measure was characterized in \cite{belinschi2009spectral}, which is a heavy-tail analogue of the
Marchenko-Pastur distribution and, in contrast to the Marchenko-Pastur law, has a nonzero mass in any neighborhood of zero. Not much has been known since then.
A recent investigation \cite{bao2024phase2} identified a change from localized to delocalized eigenvector statistics as we switch from $\alpha\in(0,2)$ to $\alpha>2$, but the asymptotics of $\sigma_{min}(X)$ remain elusive. Meanwhile, the proof methods in these existing works (in particular on squared Lévy matrices \cite{MR4310816},\cite{MR4260468},\cite{louvaris2022universality}) do not easily generalize to this heavy-tail rectangular setting.

\subsection{Two-side least singular value estimates}

The main contribution of this paper is a lower bound for $\sigma_{min}(X)$ that is asymptotically sharp up to a log factor.

\begin{theorem}
    \label{theorem1.56}
     Fix $\alpha\in(0,2),$ $\delta>1$, and let $N=\lceil \delta n\rceil$. Let $\xi$ be a random variable with a symmetric distribution $\xi \overset{\operatorname{law}}{=}-\xi$ satisfying that, for some $C_u>C_\ell>0$ and some $T_0>0$, we have that for any $t\geq T_0$,
\begin{equation}\label{stablelaw}
C_\ell t^{-\alpha}\leq \mathbb{P}(|\xi|\geq t)\leq C_ut^{-\alpha}.
\end{equation}  Let $X=(x_{ij})_{1\leq i\leq N,1\leq j\leq n}$ be a rectangular random matrix with independent identically distributed (i.i.d.) entries having the distribution $\xi$. 
Then for any $D>0$ we can find $C_1,C_2>0$ and some $N_0\in\mathbb{N}_+$ such that, whenever $N\geq N_0$,
\begin{equation}\label{twosideestimate1}
\mathbb{P}\left(C_1n^\frac{1}{\alpha}(\log n)^\frac{2(\alpha-2)}{\alpha}\leq \sigma_{min}(X)\leq C_2n^\frac{1}{\alpha}(\log n)^\frac{\alpha-2}{2\alpha}\right)\geq 1-n^{-D},
\end{equation}
 where $C_1,C_2$ depend only on $C_\ell,C_u,\alpha,\delta,D$, while
$N_0$ may additionally depend on the tail threshold $T_0$ in
\eqref{stablelaw}.
\end{theorem}

The upper bound stated in Theorem \ref{theorem1.56} was established in a recent paper by Bao,
Lee and Xu \cite{bao2024phase2}. We give a self-contained small-column proof of this bound, using the same truncation idea. The authors asked in \cite{bao2024phase2} whether a matching lower bound for $\sigma_{min}(X)$ can be derived, and in this paper we partially answer this problem by deriving a lower bound which almost matches the upper bound except for a log factor.

The proof of Theorem \ref{theorem1.56} and other main results in this work are based on a surprising interplay between two distinctive powerful techniques: a recent matrix universality principle by Brailovskaya and Van Handel \cite{brailovskaya2024universality}, and a careful adaptation of an earlier argument by Tikhomirov \cite{tikhomirov2016smallest}. Such a successful combined usage seems completely new: we need to consider both anti-concentration phenomena and the increased variance due to heavy-tail, and negligence on either side would not lead to our almost optimal estimate in \eqref{twosideestimate1}.

Recall that a random variable $Z$ is said to be $(0,\sigma)$ $\alpha$-stable for some $\alpha\in(0,2)$ and $\sigma>0$, if
$$
\mathbb{E}[e^{itZ}]=\exp(-\sigma^\alpha|t|^\alpha)\text{ for all }t\in\mathbb{R}.
$$
By standard properties of stable distribution, the estimate \eqref{stablelaw} is satisfied by all $\alpha$-stable laws for all $\alpha\in(0,2)$.

Another example illustrates which tail exponent governs the scale under
independent random mixing. Let $\mu_1,\mu_2$ be symmetric laws with
tails comparable to $t^{-\alpha_1}$ and $t^{-\alpha_2}$, where
$0<\alpha_1<\alpha_2<2$. For fixed $p\in(0,1)$, the mixture
$p\mu_1+(1-p)\mu_2$ has tail comparable to $t^{-\alpha_1}$.
Thus, for i.i.d. entries sampled from this mixture, Theorem
\ref{theorem1.56} gives the leading power $n^{1/\alpha_1}$ for
$\sigma_{min}(X)$, up to logarithmic factors. This observation concerns
independent random mixing, not a prescribed arrangement of the two types.

The lower-bound argument also applies to symmetric random variables
with regularly varying tails, after choosing a distribution-dependent
truncation level and normalization. More precisely, we have the following.

\begin{theorem}\label{slowvaryingtails}
Let $\zeta$ have a symmetric distribution and suppose that, for some
$\alpha\in(0,2)$ and some positive slowly varying function $L$,
\begin{equation}\label{forallsymmetric}
\mathbb{P}(|\zeta|\geq t)=L(t)t^{-\alpha},\qquad t\geq1.
\end{equation}
Fix $\delta>1$, set $N=\lfloor\delta n\rfloor$, and let
$X\in\mathbb{R}^{N\times n}$ have i.i.d. entries of distribution $\zeta$.
For every $D>0$, there exist a positive slowly varying function $L_D'$
and an integer $n_0$, depending on $D,\delta$ and the law of $\zeta$,
such that, for all $n\geq n_0$,
    $$
\mathbb{P}(\sigma_{min}(X)\geq n^{\frac{1}{\alpha}}L_D'(n)^{-1})\geq 1-n^{-D}.
    $$
\end{theorem}

\begin{unnumberedremark}
Theorem \ref{theorem1.56} shows that typically the smallest singular value $\sigma_{min}(X)$ is concentrated in a very small window. We believe that the lower bound in \eqref{twosideestimate1} can be improved to $n^\frac{1}{\alpha}(\log n)^\frac{\alpha-2}{2\alpha}$, i.e. the upper bound should give the optimal magnitude. Currently the discrepancy of the upper and lower bound by a log factor is due to the technique we use, see Remark \ref{majorremarks}, (4) for more details. Removing the log factor discrepancy may require considerable effort and innovative ideas, which we leave for future work. Despite this mismatch of log factors, Theorem \ref{theorem1.56} appears to be the first time that we can capture the effect of heavy tails on $\sigma_{min}(X)$ that is not existing for Gaussian distribution. Another major advantage is that the proof is robust to perturbations, even without any control on operator norm of perturbation, see Theorem \ref{lowergeneral}. A further open problem is to derive the exact distribution of $\sigma_{min}(X)$, but we feel this problem is far beyond the scope of existing techniques.

While the estimate \eqref{twosideestimate1} holds with polynomially good probability $1-O(n^{-D})$ for any $D>0$, it does not hold with exponentially good probability $1-e^{-\Omega(n)}$. Indeed, fix $c\in(0,1)$ and put $b_n=n^{(1-c)/\alpha}$. By \eqref{stablelaw}, the event $\mathcal E_n:=\{\max_{i\in[N]}|x_{i1}|<b_n\}$ has probability at least $e^{-Kn^c}$ for some constant $K>0$ and all sufficiently large $n$. Conditionally on $\mathcal E_n$, the expected squared norm of the first column is $O(nb_n^{2-\alpha})=o(n^{2/\alpha}(\log n)^{4(\alpha-2)/\alpha})$. Markov's inequality therefore shows that this column norm is below the lower threshold in \eqref{twosideestimate1} with conditional probability tending to one. Since $\sigma_{min}(X)\leq\|X_{\cdot1}\|_2$, the failure probability is at least $(1-o(1))e^{-Kn^c}$, ruling out an exponentially small failure probability. For dense random matrices with uniformly anti-concentrated entry distribution, we often have least singular value bounds with exponentially good probability, see \cite{litvak2005smallest}, \cite{rudelson2009smallest}, \cite{livshyts2021smallest}.  For sparse random matrices, small singular value estimates usually hold with probability $1-e^{-\Omega(np)}$ where $p$ is the sparsity, see \cite{basak2017invertibility}, \cite{gotze2023largest}, \cite{dumitriu2024extreme}. This is no surprise given that heavy-tail random matrix ensembles share some common features with sparse matrix ensembles.
\end{unnumberedremark}

\subsection{Diverging aspect ratio and general perturbations}

The lower bound in Theorem \ref{theorem1.56}
is valid in a much wider generality, as it continues to hold when the random variables are no longer i.i.d. and the matrix is perturbed by some other deterministic matrix $B$. It is also important to consider unbounded aspect ratio $N/n$, which we discuss here.

Lower bounds for rectangular random matrices in the sub-Gaussian setting were obtained by Rudelson and Vershynin \cite{rudelson2009smallest}, while Tikhomirov obtained estimates uniform over arbitrary deterministic perturbations under only anti-concentration assumptions \cite{tikhomirov2016smallest}. Thus, the point of Theorem \ref{lowergeneral} is not the presence of $B$ alone, but the heavy-tail scale together with the inhomogeneous coefficients and the possibly unbounded aspect ratio.

In the next theorem, we (1)consider possibly unbounded aspect ratio $N/n$; (2)remove the i.i.d. assumption and (3)consider a general perturbation $B$.

\begin{theorem}\label{lowergeneral}
    Fix $\alpha\in(0,2),$ $\delta>1$, and let $N,n$ be integers such that $N\geq \delta n$. Let $(x_{ij})_{1\leq i\leq N,1\leq j\leq n}$ be i.i.d. random variables with a symmetric distribution and satisfy \eqref{stablelaw}. Fix $0<A_1<A_2$ and let $(a_{ij})_{1\leq i\leq N,1\leq j\leq n}$ be fixed real numbers such that $$A_1\leq a_{ij}\leq A_2\quad \forall i=1,\cdots,N, j=1,\cdots,n.$$ 
    Let $A=(a_{ij}x_{ij})$ be an $N\times n$ random matrix
and let $B$ be any deterministic $N\times n$ matrix. We do not make any assumption on $B$.
Then for any $D>0$ we can find $C_3>0$ and $N_0\in\mathbb{N}_+$ such that,
whenever $N\geq N_0,$ 
\begin{equation}\label{onesideestimate}
\mathbb{P}(\sigma_{min}(A+B)\leq C_3N^\frac{1}{\alpha}(\log N)^\frac{2(\alpha-2)}{\alpha})\leq N^{-D},
\end{equation}
 where $C_3$ depends only on $D,\delta,\alpha,C_\ell,C_u,A_1,A_2$,
while $N_0$ may additionally depend on the tail threshold $T_0$ in
\eqref{stablelaw}. Neither constant depends on $B$.
\end{theorem}

In Theorem \ref{theorem1.56} we assumed a bounded aspect ratio $N/n$ and do not aim for a sharp estimate on the leading constant, so we can freely interchange $n$ by $N$, or vice versa, in estimate \eqref{twosideestimate1}. But in Theorem \ref{lowergeneral} the aspect ratio is unbounded, so we need to be careful that it is the factor $N$, rather than $n$, appearing in \eqref{onesideestimate}.

An upper bound for $\sigma_{min}(X)$ can also be proven for diverging $N/n$.

\begin{Proposition}\label{propositionunbounded} (Upper bound, diverging $N/n$) Fix $\alpha\in(0,2)$, $\delta>1$ and let integers $N,n$ satisfy $N\geq\delta n$ and $n\geq2$. Consider an $N\times n$ random matrix $X=(x_{ij})$ with i.i.d. symmetric entries satisfying \eqref{stablelaw}.
Then for any $D>0$ and $\mathbf b\in(0,1)$ there are $C_4>0$ and $N_0\in\mathbb N$ such that, for $N\geq N_0$,
$$
\mathbb{P}(\sigma_{min}(X)\geq C_4N^\frac{1}{\alpha}(\log n)^\frac{\alpha-2}{2\alpha})\leq N^{-D}+\exp(-n^\mathbf{b}),
$$ where $C_4$ depends only on $\alpha,C_u,\mathbf b$, and $N_0$ depends additionally on the tail threshold $T_0$ in \eqref{stablelaw}.
\end{Proposition}

Although Proposition \ref{propositionunbounded} is stated for i.i.d. entries, its small-column proof does not rely on an i.i.d. operator-norm estimate. It uses independence and uniform upper-tail and truncated-second-moment bounds. In particular, for $A=(a_{ij}x_{ij})$ with deterministic $|a_{ij}|\leq A_2$, every column satisfies $\|A_{\cdot j}\|_2\leq A_2\|X_{\cdot j}\|_2$. The same proof therefore gives the stated upper bound for $\sigma_{min}(A)$ after multiplying $C_4$ by $A_2$. No upper bound uniform over an arbitrary deterministic perturbation $B$ is asserted.

We stress that the upper bound of Proposition \ref{propositionunbounded} involves a factor $(\log n)$ but the lower bound of Theorem \ref{lowergeneral} involves a much larger factor $(\log N)$. They cannot be interchanged when $N/n$ is very large. The reason is, for the lower bound we need a form of global universality principle, and $N$ is the typical size of the rectangular matrix $X$. For the upper bound we only need to find some columns of $X$ with certain good properties, and there are only $n$ columns, so we only need factor $\log n$. From this reasoning, we believe that in the general case of unbounded $\log N/\log n$, the idea of this work will never lead to lower and upper bounds of $\sigma_{min}(X)$ having the same order of magnitude.

\subsection{A geometric application: Euclidean balls inside random polytopes}

For an $N\times n$ matrix $\Gamma$ with rows $v_i^*$, set
$K_N:=\Gamma^*(B_1^N)=\operatorname{conv}\{\pm v_i:i\in[N]\}$,
where $*$ denotes transpose and $B_p^n$ is the unit $\ell_p^n$ ball,
$1\leq p\leq\infty$.
For standard Gaussian $\Gamma$, fixed $\delta>1$, $N\geq\delta n$,
and sufficiently large $n$, there are $\epsilon(\delta),c(\delta)>0$
such that
\begin{equation}\label{l1quotient}
K_N\supset\epsilon(\delta)B_2^n
\end{equation}
with probability at least $1-e^{-c(\delta)n}$; see
\cite{litvak2002randomized,mankiewicz2000compact,wojtaszczyk2010stability}
and \cite[Corollary 4.1]{litvak2005smallest}.
This $\ell_1$ quotient property has applications to compressive sensing
\cite{donoho2006compressed,candes2006stable,candes2006compressive}.

For independent symmetric variance-one sub-Gaussian entries,
\cite[Theorem 4.2]{litvak2005smallest} gives
$K_N\supset\frac18(B_\infty^n\cap\sqrt{C_1\ln(N/n)}B_2^n)$
with probability at least $1-2e^{-cn}$ when $C_2n\leq N\leq2^n$;
the positive constants depend only on the common sub-Gaussian bound.
Theorem 1.1 of \cite{guedon2020random} extends this inclusion to
independent symmetric variance-one entries with uniform anti-concentration
and row-wise identical distributions, for $N/n$ sufficiently large in
terms of the anti-concentration parameters, without higher moments.
Related work includes \cite{hayakawa2023estimating}.

Theorem \ref{lowergeneral} yields a heavy-tail Euclidean inradius bound.

\begin{corollary}\label{corollary}
    Fix $\delta>1$, $\alpha\in(0,2)$ and let integers $N,n$ satisfy $N\geq \delta n$. Let $X=(x_{ij})$ be an $N\times n$ rectangular matrix with i.i.d. symmetric entries satisfying \eqref{stablelaw} for all $t\geq T_0$. Then for any $D>0$ we can find $N_0\in\mathbb{N}$ and $C_4>0$ such that whenever $N\geq \max(N_0,\delta n)$, with probability at least $1-N^{-D}$ we have
    $$
X^*(B_1^N)\supset C_4 N^{\frac{1}{\alpha}-\frac{1}{2}}(\log N)^\frac{2(\alpha-2)}{\alpha}B_2^n,
    $$
    where $C_4$ depends only on $D,\alpha,\delta,C_\ell,C_u$, and
    $N_0$ may additionally depend on the tail threshold $T_0$.
\end{corollary}

For $\delta\leq N/n\leq T<\infty$, this radius diverges for every
$\alpha\in(0,2)$, whereas the Gaussian Euclidean inradius has constant order
\cite[Corollary 4.1 and Theorem 4.9]{litvak2005smallest}.
For standard symmetric $\alpha$-stable entries, $1\leq\alpha<2$,
\cite[Theorem 1.8]{guedon2022geometry} gives
$X^*(B_1^N)\supset c(\alpha)(N/n)^{\vartheta/\alpha}B_q^n$,
where $q=\alpha/(\alpha-1)$ for $\alpha>1$ and $q=\infty$ for $\alpha=1$.
This holds for $\vartheta\in(0,1)$ and $N\geq C(\vartheta,\alpha)n$,
with probability at least $1-2e^{-cN^{1-\vartheta}n^\vartheta}$.
Since $B_q^n$ has Euclidean inradius one, the Euclidean-radius lower bound
obtained from this inclusion is independent of $n$ at a fixed aspect
ratio, whereas the lower bound in Corollary \ref{corollary} tends to infinity.

\begin{proof}[\proofname\ of Corollary \ref{corollary}]
For $K=X^*(B_1^N)$, its support function satisfies, for every
$z\in\mathbb{R}^n$,
$$
h_K(z)=\|Xz\|_\infty
\geq\frac{\|Xz\|_2}{\sqrt N}
\geq\frac{\sigma_{min}(X)}{\sqrt N}\|z\|_2.
$$
The support-function inclusion criterion
\cite[Section 11]{rockafellar1970convex} gives
$K\supset(\sigma_{min}(X)/\sqrt N)B_2^n$.
Applying Theorem \ref{lowergeneral} with $a_{ij}=1$ and $B=0$
concludes the proof.
\end{proof}

\section{Proof outline of main results}

To study the heavy-tail random matrix $X$, an (by now standard, see \cite{aggarwal2019bulk}, \cite{MR4310816},\cite{louvaris2022universality},\cite{bao2024phase},\cite{han2024deformed}) approach is to resample the entries of $X$. Specifically,  we first assign a label taking value in $\{0,1\}$ to each entry, where a label 1 is assigned to site $(i,j)$ meaning that $x_{ij}$ takes a small value, and value 0 assigned to $(i,j)$ meaning that $x_{ij}$ takes a large value.  Then we sample the small and large entry parts of $X$ separately via conditional laws of $x_{ij}$ conditioned to be small or large.  Conditioned on the label, the entries of $X$ remain independent.

\subsection{Resampling for bounded aspect ratio}

We outline this resampling procedure for $X$. In this section we assume the aspect ratio is bounded: for some fixed $1<\delta<T<\infty$,
$$
\delta n\leq N\leq Tn.
$$
The diverging aspect ratio case is dealt in Section \ref{divergingaspect}.

\begin{Definition} The sampling procedure (for bounded aspect ratio) has the following steps:
    
\end{Definition}
\begin{enumerate}
\item Fix $\tilde{\epsilon}_n\in(0,\frac{1}{\alpha})$ whose value depends on $n$ and will be determined later. Different values of $\tilde{\epsilon}_n$ will be used in each specific application.

\item Consider $\Psi=(\psi_{ij})$ an $N\times n$ random matrix with independent Bernoulli entries distributed as 
    \begin{equation}\label{generationoflabels}
\psi_{ij}=\begin{cases}1,\quad\text{with probability } \mathbb{P}(|x_{ij}|< n^{\frac{1}{\alpha}-\tilde{\epsilon}_n}),\\0,\quad
\text{with probability } \mathbb{P}(|x_{ij}|\geq n^{\frac{1}{\alpha}-\tilde{\epsilon}_n}).
\end{cases}
    \end{equation}

\item Define two classes of random variables $y_{ij}$ and $z_{ij}$, that are conditional versions of $x_{ij}$ such that $|x_{ij}|< n^{\frac{1}{\alpha}-\tilde{\epsilon}_n}$ and $|x_{ij}|\geq n^{\frac{1}{\alpha}-\tilde{\epsilon}_n}$. More precisely, we define $y_{ij}$ and $z_{ij}$ to satisfy, for any interval $\tilde{I}\subset\mathbb{R}$,
\begin{equation}
\mathbb{P}(y_{ij}\in\tilde{I})=\frac{\mathbb{P}\left(x_{ij}\in\tilde{I}\cap(-n^{\frac{1}{\alpha}-\tilde{\epsilon}_n},n^{\frac{1}{\alpha}-\tilde{\epsilon}_n})\right)}{\mathbb{P}(x_{ij}\in(-n^{\frac{1}{\alpha}-\tilde{\epsilon}_n},n^{\frac{1}{\alpha}-\tilde{\epsilon}_n}))},
\end{equation}

\begin{equation}\label{lawofz}
\mathbb{P}(z_{ij}\in\tilde{I})=\frac{\mathbb{P}(x_{ij}\in\tilde{I}\cap((-\infty,-n^{\frac{1}{\alpha}-\tilde{\epsilon}_n}]\cup[n^{\frac{1}{\alpha}-\tilde{\epsilon}_n},\infty)))}{\mathbb{P}(x_{ij}\in(-\infty,-n^{\frac{1}{\alpha}-\tilde{\epsilon}_n}]\cup[n^{\frac{1}{\alpha}-\tilde{\epsilon}_n},\infty))}.
\end{equation}  \item It is not hard to check that the matrix $X$ has the same distribution as 
\begin{equation}\label{resamplingdecision}X\overset{\operatorname{law}}{=} {\Psi}\circ Y+(\mathbf{1}-{\Psi})\circ Z,
\end{equation}
where $Y=(y_{ij})_{1\leq i\leq N,1\leq j\leq n}$ and $Z=(z_{ij})_{1\leq i\leq N,1\leq j\leq n}$, and entries of $Y$ and $Z$ are independent. The symbol $\circ$ denotes the entry-wise (Hadamard) product of two $N\times n$ matrices: $(A\circ B)_{ij}=a_{ij}b_{ij}$, and $\mathbf{1}$ denotes the all-ones $N\times n$ matrix.
\end{enumerate}

\subsection{The upper bound}
The upper bound in Theorem \ref{theorem1.56} was established in
\cite[Theorem 3.10]{bao2024phase2}. We give a small-column proof that
also applies when the aspect ratio is unbounded.

\begin{theorem}\label{theorem2.22}[\cite{bao2024phase2}, Theorem 3.10]
Fix $\alpha\in(0,2)$, $\delta>1$ and $N=\lceil\delta n\rceil$.
Let $X$ have i.i.d. symmetric entries satisfying \eqref{stablelaw}.
For any $D>0$ there are $C_\eqref{theorem2.22}>0$ and
$N_\eqref{theorem2.22}\in\mathbb N$ such that
$$
\mathbb{P}(\sigma_{min}(X)\geq C_\eqref{theorem2.22}n^{\frac{1}{\alpha}}(\log n)^{\frac{\alpha-2}{2\alpha}})\leq n^{-D},\quad N\geq N_\eqref{theorem2.22}.
$$
The constant $C_\eqref{theorem2.22}$ depends only on
$\alpha,\delta,C_u$, while $N_\eqref{theorem2.22}$ depends additionally
on $D$ and the tail threshold $T_0$.
\end{theorem}
\begin{proof}
Apply Lemma \ref{smallcolumnupper} below with $\mathbf b=1/2$.
Since $N\leq(\delta+1)n$, its threshold is at most
$C_{\rm up}(\delta+1)^{1/\alpha}
n^{1/\alpha}(\log n)^{(\alpha-2)/(2\alpha)}$.
Set $C_\eqref{theorem2.22}=C_{\rm up}(\delta+1)^{1/\alpha}$ and
increase $N_\eqref{theorem2.22}$ so that the lemma applies, $n\geq2$,
and $e^{-\sqrt n}\leq n^{-D}$.
\end{proof}

Theorem \ref{theorem2.22} is the bounded-aspect-ratio consequence of the
following small-column estimate. We state this estimate separately
because it requires only an upper-tail bound, without symmetry or a
matching lower-tail bound, and holds uniformly for all $N\geq n\geq2$.
The same lemma also yields Proposition \ref{propositionunbounded} for
unbounded aspect ratios.

\begin{lemma}\label{smallcolumnupper}
Let $N\geq n\geq2$, and let $X$ have i.i.d. real entries satisfying
$\mathbb P(|x_{ij}|\geq t)\leq C_ut^{-\alpha}$ for $t\geq T_0$,
where $\alpha\in(0,2)$. For every $\mathbf b\in(0,1)$, there exist
$C_{\rm up}>0$ and $N_{\rm up}\in\mathbb N$ such that, for
$N\geq N_{\rm up}$,
$$
\mathbb P\!\left\{\sigma_{min}(X)\geq
C_{\rm up}N^{1/\alpha}(\log n)^{(\alpha-2)/(2\alpha)}\right\}
\leq e^{-n^{\mathbf b}}.
$$
Here $C_{\rm up}$ depends only on $\alpha,C_u,\mathbf b$, and
$N_{\rm up}$ depends additionally on $T_0$. Symmetry is not required.
\end{lemma}
\begin{proof}
Write $R=|x_{11}|$. Integrating the upper-tail bound and absorbing
$T_0^2$ into $B^{2-\alpha}$ gives, for $B\geq b_0(\alpha,C_u,T_0)$,
$$
p:=\mathbb P(R\geq B)\leq\tfrac12,\qquad
\mathbb E[R^2\mid R<B]\leq K_0B^{2-\alpha},
\quad K_0:=2+\frac{4C_u}{2-\alpha}.
$$
Set
$$
\kappa:=\frac{1-\mathbf b}{2},\qquad
\eta:=1-2^{-\kappa},\qquad
B:=\left(\frac{2C_uN}{\kappa\log n}\right)^{1/\alpha}.
$$
For sufficiently large $N$, the preceding bounds hold uniformly over
$2\leq n\leq N$, since
$B^\alpha\geq2C_uN/(\kappa\log N)\to\infty$.
For $\mathcal E_j:=\{|x_{ij}|<B\text{ for every }i\in[N]\}$,
$\log(1-p)\geq-2p$ gives
$$
\mathbb P(\mathcal E_j)=(1-p)^N
\geq e^{-2Np}\geq n^{-\kappa}.
$$
Set $L:=\sqrt{K_0/\eta}\,\sqrt N\,B^{1-\alpha/2}$.
Conditional Markov's inequality yields
$$
\mathbb P\{\|X_{\cdot j}\|_2<L\}
\geq(1-\eta)\mathbb P(\mathcal E_j)
\geq2^{-\kappa}n^{-\kappa}.
$$
The columns are independent and
$\sigma_{min}(X)\leq\min_j\|X_{\cdot j}\|_2$, so
$$
\mathbb P\{\sigma_{min}(X)\geq L\}
\leq\exp(-2^{-\kappa}n^{1-\kappa})\leq\exp(-n^{\mathbf b}),
$$
because $1-\kappa-\mathbf b=\kappa$ and $n\geq2$.
Substituting $B$ into $L$ gives the claimed threshold with
$C_{\rm up}$ depending only on $\alpha,C_u,\mathbf b$.
\end{proof}

For bounded $N/n$, the truncation level in this proof can be written
$B=n^{1/\alpha-\tilde\epsilon_n}$, where
\begin{equation}\label{choiceofepsilonn}
n^{\alpha\tilde\epsilon_n}
=\frac{\kappa n\log n}{2C_uN}\asymp\log n.
\end{equation}

\subsection{The lower bound}

For the lower bound, we use the same decomposition
$X=\Psi\circ Y+(\mathbf{1}-\Psi)\circ Z$, but choose the truncation
level differently. Put $R=|x_{ij}|$ and $\bar F(t)=\mathbb{P}(R\geq t)$.
For every $b>0$, the exact truncated-moment identity is
\begin{equation}\label{truncatedsecondmoment}
\mathbb{E}[R^2\mathbf{1}_{\{R<b\}}]
=2\int_0^b t\bar F(t)\,dt-b^2\bar F(b).
\end{equation}
This identity remains valid when the distribution has an atom at $b$.
For the upper bound, \eqref{stablelaw} gives, for $b\geq T_0$,
$$
\mathbb{E}[R^2\mathbf{1}_{\{R<b\}}]
\leq T_0^2+\frac{2C_u}{2-\alpha}b^{2-\alpha}.
$$
For a positive lower bound, set
$\theta=(C_\ell/(2C_u))^{1/\alpha}\in(0,1)$.
Whenever $b\geq T_0/\theta$,
$$
\begin{aligned}
\mathbb{E}[R^2\mathbf{1}_{\{R<b\}}]
&\geq\theta^2b^2\mathbb{P}(\theta b\leq R<b)\\
&\geq\theta^2b^2(C_\ell\theta^{-\alpha}-C_u)b^{-\alpha}
=\theta^2C_u b^{2-\alpha}.
\end{aligned}
$$
More explicitly, set
$$
B_0:=\max\left\{1,\frac{T_0}{\theta},(2C_u)^{1/\alpha},
T_0^{2/(2-\alpha)}\right\},\qquad
c_0:=\theta^2C_u,\qquad C_0:=2+\frac{4C_u}{2-\alpha}.
$$
For $b\geq B_0$, we have $\bar F(b)\leq1/2$ and
$T_0^2\leq b^{2-\alpha}$. Division by
$\mathbb{P}(R<b)=1-\bar F(b)$ therefore gives
$$
c_0b^{2-\alpha}
\leq\mathbb{E}[R^2\mid R<b]
\leq C_0b^{2-\alpha}.
$$
The constants $c_0,C_0$ depend only on $\alpha,C_\ell,C_u$, whereas
the starting point $B_0$ also depends on $T_0$.
Apply this with $b=n^{1/\alpha-\tilde\epsilon_n}$, where
$\tilde\epsilon_n=o(1)$. Symmetry gives $\mathbb{E}y_{ij}=0$, and
the normalization yields
$$
\mathbb{E}[|n^{-\frac{1}{\alpha}+(1-\frac{\alpha}{2})\tilde{\epsilon}_n}y_{ij}|^2]=C_n\cdot n^{-1}
$$
with $c_0\leq C_n\leq C_0$ for all sufficiently large $n$.
No convergence of $C_n$ is asserted under the two-sided tail bounds.
We summarize these facts as follows.
\begin{lemma}\label{lemma2.34}
Let $\xi_n$ have the distribution of
$n^{-1/\alpha+(1-\alpha/2)\tilde\epsilon_n}y_{ij}$. Then, for all
sufficiently large $n$,
$$
\mathbb{E}\xi_n=0,\qquad
\mathbb{E}|\xi_n|^2=C_n/n,\quad c_0\leq C_n\leq C_0,\qquad
|\xi_n|\leq n^{-\alpha\tilde\epsilon_n/2}\quad\text{a.s.}
$$
\end{lemma}
Instead of choosing $\tilde{\epsilon}_n$ as in \eqref{choiceofepsilonn}, we will choose $\tilde{\epsilon}_n$ in the proof of lower bound as
\begin{equation}\label{choicelowerbound}
n^{\alpha\tilde{\epsilon}_n}=\mathbf{c}(\log n)^4
\end{equation} where $\mathbf{c}>0$ will later be set to be a sufficiently large constant. This change is made for a purely technical purpose, so that we can use Theorem \ref{derivation2.414}. See Remark \ref{majorremarks}, (4) for more discussions on the optimality of $\tilde{\epsilon}_n$.

We let $\mathcal{F}_\Psi$ be a sigma subalgebra of the probability space $(\Omega,\mathcal{F},\mathbb{P})$ spanned by all the random variables $\psi_{ij}$ and the events $\{\mathbf{\psi}=W\}$ for any $W\in\{0,1\}^{N\times n}$. Then conditioning on the filtration $\mathcal{F}_\Psi$, the random variables $y_{ij}$ and $z_{ij}$ are mutually independent.

Now we outline a heuristic argument to prove Theorem \ref{theorem1.56}.
For any fixed choice of label ${\Psi}$, we first sample the entries $Z=(z_{ij})$, and we show that no matter what value $z_{ij}$ takes, we can always use the randomness of $y_{ij}$ to derive high probability lower bounds for $\sigma_{min}(X)$ that are independent of $Z$.

Via Lemma \ref{lemma2.34}, the additional factor $C_n^{-1/2}$ makes
the truncated entry variance equal to $n^{-1}$. A naive application
of the classical Bai--Yin theorem (Theorem \ref{theorem1.111}) would
therefore suggest that
$\sigma_{min}(C_n^{-1/2}n^{-1/\alpha+(1-\alpha/2)\tilde\epsilon_n}X)$
is close to $\sqrt\delta-1$.

This heuristic argument has several flaws. First, even the unit-variance variables $\eta_n:=\sqrt{n/C_n}\,\xi_n$ are not uniformly anti-concentrated. Since $\alpha<2$, $\tilde\epsilon_n=o(1)$ and $C_n\geq c_0>0$, the multiplier relating $\eta_n$ to $y_{ij}$ tends to zero, while the conditional laws of $y_{ij}$ are tight. Thus $\eta_n\to0$ in probability, so $\mathbb P(|\eta_n|\leq r)\to1$ for every fixed $r>0$. Theorem \ref{theorem1.2345} therefore cannot directly give a uniform lower bound at this normalization. Second, by our sampling procedure there are still $\gg n$ entries $(i,j)$ of $X$ where $\psi_{ij}=0$ and thus the randomness at site $(i,j)$ cannot be used directly. With high possibility each row has such entries, so we have to deal with $X$ with some randomness removed. Third, we have very weak control on the operator norm of the conditionally non-random part $\|(\mathbf{1}-\Psi)\circ Z\|$ containing the large entries, and we need a method to lower bound the smallest singular value for a rectangular matrix with normalized entries perturbed by a large deterministic matrix. This is very similar to the setting of \cite{tikhomirov2016smallest}.

\subsection{A two-moment replacement principle for minimal singular value}

We now address the first caveat raised in the last section. After fixing the label and the large-entry part, we compare the normalized matrix with a Gaussian matrix having the same entry variances. In terms of the unit-variance variables $\eta_n=\sqrt{n/C_n}\,\xi_n$, the replacement is by a standard real Gaussian $\gamma$; equivalently, $\xi_n$ is replaced by $\sqrt{C_n/n}\,\gamma$.

The next theorem is essentially a corollary of \cite{brailovskaya2024universality}, Theorem 2.6 and 3.16. 
\begin{theorem}\label{derivation2.414} Let $N\geq n$.
    Let $T$ be an $N\times n$ matrix with decomposition $T=T_0+T_1$, where $T_0$ is deterministic, $T_1$ has independent, mean zero coordinates $(t_{ij})$ such that $\sup_{i,j}\mathbb{E}[|t_{ij}|^2]\leq \frac{M^2}{n}$ for some fixed $M>1$ and let $q>0$ be such that  $$\sup_{i,j}|t_{ij}|\leq q,\quad a.s.$$ Let $G=G_0+G_1$ be another $N\times n$ matrix where $G_0=T_0$, $G_1=(g_{ij})$ has independent mean zero Gaussian entries and share the same variance profile as $T_1$ i.e. $\mathbb{E}[g_{ij}^2]=\mathbb{E}[t_{ij}^2]$ for each $i,j$. Then on a common probability space supporting an independent copy of $G$ and $T$, we have: for any $t>0$,
     $$
\mathbb{P}(|\sigma_{min}(T)-\sigma_{min}(G)|\geq Cn^{-\frac{1}{2}}t^{\frac{1}{2}}+Cq^{\frac{1}{3}}t^{\frac{2}{3}}(\frac{N}{n})^\frac{1}{3}+Cqt)\leq 8N e^{-t},
    $$ where $C>0$ is a universal constant that only depends on $M>0$.
\end{theorem}

When the aspect ratio $N/n$ is bounded, we take $t=(D+2)\log N\asymp\log n$ so that the failure probability is polynomially small. To make the comparison error smaller than the fixed positive lower bound for $\sigma_{min}(G)$ obtained below, it suffices to take $q\leq c(\log n)^{-2}$ with $c>0$ sufficiently small: this controls $q^{1/3}t^{2/3}(N/n)^{1/3}$, while the other two terms vanish as $n\to\infty$. This is the reason we set \eqref{choicelowerbound}. The proof of Theorem \ref{derivation2.414} is deferred to Section \ref{section3.5}. When the aspect ratio $N/n$ is diverging, we will set the magnitude of $q$ differently, see Section \ref{divergingaspect}.

\begin{remark}\label{majorremarks}
We give several important remarks before applying this theorem.
\begin{enumerate} \item Theorem \ref{derivation2.414} does not require that entries of $T_1$, or any of its normalized version, are uniformly anti-concentrated. \item  All parameters in Theorem \ref{derivation2.414} depend only on the random part $T_1$ and do not depend on the deterministic part $T_0$. This is very important for our application. \item 
We can also compare $T$ to its free probability version $T_{\text{free}}$ and compute $\sigma_{min}(T_{\text{free}})$ using Lehner's formula \cite{lehner1999computing}. This innovative approach is proposed and developed carefully in \cite{bandeira2023matrix}. However, we have almost no control over the very large deterministic part $T_0$, and some randomness is removed from the random part $T_1$. This means that the computation of $\sigma_{min}(T_{\text{free}})$ will be very difficult. Instead, we take the other approach by lower bounding $\sigma_{min}(G)$, which draws us on the techniques in the subfield of quantitative invertibility of random matrices (see review \cite{MR4680362}) and in particular in a very similar situation as \cite{tikhomirov2016smallest}. \item  The choice of $\tilde{\epsilon}_n$ in \eqref{choicelowerbound} is made only for Theorem \ref{derivation2.414} to produce meaningful quantitative estimates. This choice of $\tilde{\epsilon}_n$ is the source of the mismatch in Theorem \ref{theorem1.56}: the lower bound is smaller than the upper bound by a factor of $(\log n)^\frac{3(2-\alpha)}{2\alpha}$. We expect the choice $n^{\alpha\tilde{\epsilon}_n}\asymp\log n$, equivalently $q\asymp(\log n)^{-\frac{1}{2}}$, as in \eqref{choiceofepsilonn}, to be optimal for proving a similar comparison theorem to Theorem \ref{derivation2.414}. Currently, there exists another method to bound the least singular value of a sparse rectangular matrix, which is \cite{dumitriu2024extreme}. The work \cite{dumitriu2024extreme} remarkably achieves the optimal sparsity scale for a Bai-Yin type result to hold, but the proof requires the entries to have zero mean. Informally, the method of Theorem \ref{derivation2.414}
works only for inhomogeneous sparse matrix models with sparsity (suitably defined) at least $(\log n)^4/n$,
while for the homogeneous case we can go down to sparsity $(\log n)/n$.  In our setting, considering a large chunk of fixed non-zero entries is fundamentally important, and it is unclear how to generalize the proof of \cite{dumitriu2024extreme} to a matrix with arbitrary mean, so we do not take this approach.

\end{enumerate}
\end{remark}

\subsection{When anti-concentration takes in}
Having Theorem \ref{derivation2.414} in hand, the next step is to lower bound $\sigma_{min}(G)$ for a rectangular random matrix $G$ with Gaussian entries. As some randomness of $G$ is removed, the most direct proof does not work. Also, it appears that assuming a Gaussian entry does not simplify much the proof, so we make a more general statement for independent sub-Gaussian entries with uniform anti-concentration.

Recall that a real random variable $\xi$ is said to be $K$-sub-Gaussian for some $K>0$ if 
$\mathbb{E}[\exp(|\xi|^2/K^2)]\leq 2.$

\begin{Definition}\label{line276definition}
For a given label $\Psi=(\psi_{ij})\in\{0,1\}^{N\times n}$, we define a random matrix $G$ with label $\Psi$ via $G=G_0+G_1,$ where $G_0$ is an arbitrary deterministic $N\times n$ matrix with no constraint imposed. We define  $G_1=(\psi_{ij}g_{ij})_{ij}$ where $g_{ij}$ are i.i.d. mean zero, variance one,  $K$- sub-Gaussian random variables satisfying an anti-concentration estimate: for some given $\alpha_\eqref{line276definition},\beta_\eqref{line276definition}>0,$
$$
\sup_{\lambda\in\mathbb{R}}\mathbb{P}\{|g_{ij}-\lambda|\leq\alpha_\eqref{line276definition} \}\leq 1-\beta_\eqref{line276definition}.
$$

The matrix $G$ depends on its deterministic part $G_0$ and the label ${\Psi}$, but we shall use the notation $G$ to simplify the notations and keep the dependence on ${\Psi}$ implicit.
\end{Definition}
A lower bound on $\sigma_{min}(G)$ depends on the labels ${\Psi}$, as, for example, if there are too many $\psi_{ij}=0$ so that $G$ has a zero column, then $\sigma_{min}(G)=0$. We extract a high probability event and show that most labels $\Psi$ lead to a good lower bound on $\sigma_{min}(G).$

For any deterministic label $W=(w_{ij})\in\{0,1\}^{N\times n}$, let
$\mathbb{P}_W$ denote the law obtained by fixing the mask to be $W$
and retaining the original joint law of the entry variables. In
particular, for any Borel set $B\subset\mathbb{R}^{N\times n}$, define
\begin{equation}\label{definition1w}
\mathbb{P}_W(G\in B)
:=\mathbb{P}_{g}\{G_0+(w_{ij}g_{ij})_{ij}\in B\},
\end{equation}
where $\mathbb{P}_g$ means probability over the independent entry
variables. The same convention applies to events involving the label
and entries (and auxiliary independent copies): replace the label by
$W$ and take probability over the entry variables.
When $\Psi$ is sampled independently of the entries and
$\mathbb{P}(\Psi=W)>0$, this agrees with conditioning on $\Psi=W$.
The fixed-mask law is also defined for labels of probability zero;
in particular, it allows deterministic all-one labels.

For a subset $\Omega_\Psi\subset\{0,1\}^{N\times n}$ with
$\mathbb{P}(\Psi\in\Omega_\Psi)>0$, we also use ordinary event
conditioning: for any event $\mathcal{A}\in\mathcal{F}$, set
\begin{equation}\label{definition2label}
\mathbb{P}_{\Omega_\Psi}(\mathcal{A}):=\frac{\mathbb{P}(\mathcal{A}, \Psi\in\Omega_\Psi)}{\mathbb{P}(\Psi\in\Omega_\Psi)}.
\end{equation}
We only use this event-conditional notation when its denominator is positive.

In the next theorem, the distribution of random label $\Psi$ does not need to follow \eqref{generationoflabels}.

\begin{theorem}\label{universalbacks}
    Fix any $\delta>1$,  consider $N\geq\delta n$ and let $G$ be as in Definition \eqref{line276definition}. Then there exists a subset of labels $\mathcal{D}\subset\{0,1\}^{N\times n}$ which satisfies, for a randomly generated label $\Psi$ such that each $\psi_{ij}$ is i.i.d. taking value in $\{0,1\}$ with $\mathbb{P}(\psi_{ij}=1)\geq \frac{3\delta+1}{4\delta}$, then there exists a positive integer $N_\eqref{universalbacks}\in\mathbb{N}$ such that for any $N\geq N_\eqref{universalbacks}$, 
    \begin{enumerate}
    \item For this random choice of label $\Psi$, $$\mathbb{P}(\Psi\in\mathcal{D})\geq 1-\exp(-w_\eqref{universalbacks}N).
    $$ \item For any fixed label $\Psi\in \mathcal{D}$ and any deterministic $N\times n$ matrix $G_0$, we have
    $$
\mathbb{P}_{\Psi}(\sigma_{min}(G)\leq h_\eqref{universalbacks}\sqrt{N})\leq \exp(-w_\eqref{universalbacks}N),
    $$\end{enumerate}
    where $h_\eqref{universalbacks}>0$ and
$N_\eqref{universalbacks}\in\mathbb{N}$ depend only on
$\delta$, $\alpha_\eqref{line276definition}$,
$\beta_\eqref{line276definition}$, and $K$, whereas
$w_\eqref{universalbacks}>0$ depends only on
$\delta$ and $\beta_\eqref{line276definition}$.
These constants are independent of the choice of the deterministic part
$G_0$ and the label $\Psi\in\mathcal{D}$.
\end{theorem}

The proof of Theorem \ref{universalbacks}
draws several ideas from \cite{tikhomirov2016smallest}, yet we cannot use its main result directly.
The difference is we need to fix a single label $\Psi$ a-priori and show $\sigma_{min}(G)$ is lower bounded uniformly for each $\Psi\in\mathcal{D}$, whereas \cite{tikhomirov2016smallest} does not involve this quasi-random selection procedure in its statement of main result. Nonetheless, the validity of our quasi-random selection is implicit in the proof in \cite{tikhomirov2016smallest} so that we only need to wrap up its proof in a different way. The construction of the subset $\mathcal{D}$
is inexplicit and depends on the choice of several families of $\epsilon$-nets. We will defer the proof of Theorem \ref{universalbacks} to Section \ref{section3.4theends}, and use it to now prove the main result of this paper.

\subsection{Proof of main results: bounded aspect ratio}
\label{section2.66}Now we are ready to prove the main theorems of this paper assuming Theorem \ref{derivation2.414} and Theorem \ref{universalbacks}.

\begin{proof}[\proofname\ of Theorem \ref{theorem1.56}]
Put $D_*:=D+1$. Theorem \ref{theorem2.22} gives the upper bound
with failure probability at most $n^{-D_*}$ for sufficiently large $N$.
For the lower bound, choose $\tilde\epsilon_n$ as in
\eqref{choicelowerbound} and use \eqref{resamplingdecision} to set
\begin{equation}\label{symbolT1}
T_1:= n^{-\frac{1}{\alpha}+(1-\frac{\alpha}{2})\tilde{\epsilon}_n}\Psi\circ Y,\quad T_0:= n^{-\frac{1}{\alpha}+(1-\frac{\alpha}{2})\tilde{\epsilon}_n}(\mathbf{1}-\Psi)\circ Z
,\end{equation}
so that $T:=T_0+T_1$ has the law of the normalized matrix.
Conditionally on $\Psi,Z$, the entries of $T_1$ remain independent
and $T_0$ is deterministic. Lemma \ref{lemma2.34} verifies the
assumptions of Theorem \ref{derivation2.414} with
$q=\mathbf c^{-1/2}(\log n)^{-2}$ and $M>1$ depending only on
$\alpha,C_\ell,C_u$. Taking $t=(D_*+2)\log N$ and using
$\delta\leq N/n\leq\delta+1$ gives, for fixed $\mathbf c\geq1$
and sufficiently large $N$, uniformly over labels after averaging over $Z$,
\begin{equation}\label{line498}
\mathbb P_\Psi\{ |\sigma_{min}(T)-\sigma_{min}(G)|
\geq C_{D_*}\mathbf c^{-1/6}\}\leq8Ne^{-t}\leq N^{-D_*}.
\end{equation}
Here $C_{D_*}$ depends only on $D_*,\delta,\alpha,C_\ell,C_u$, and
$G=T_0+\sqrt{C_n/n}\,(\psi_{ij}\gamma_{ij})$ is the associated
Gaussian model, with independent standard real $\gamma_{ij}$.

The active entries of $\sqrt{n/C_n}\,G$ have variance one.
Since $\mathbb P(\psi_{ij}=1)\to1$, Theorem \ref{universalbacks}
gives a set of labels $\mathcal D$, uniform in the deterministic part,
with $\mathbb P(\Psi\notin\mathcal D)\leq e^{-wN}$, where
$w=w_\eqref{universalbacks}$. Using $C_n\geq c_0$, $N/n\geq\delta$
and averaging over $Z$, we obtain, for every $\Psi\in\mathcal D$,
  \begin{equation}\label{line505}
\mathbb P_\Psi\{\sigma_{min}(G)\leq h\}\leq e^{-wN},
\qquad h:=h_\eqref{universalbacks}\sqrt{c_0\delta}>0.
  \end{equation}

Choose $\mathbf c\geq1$ so that $C_{D_*}\mathbf c^{-1/6}\leq h/2$,
then $N_0$ so that $n\geq4$, $N\geq8$, all preceding estimates hold,
and $e^{-wN}\leq N^{-D_*}$. This also ensures
$\tilde\epsilon_n\in(0,1/\alpha)$ and cutoff at least $B_0$, since
$n/(\mathbf c(\log n)^4)\to\infty$. For every $W\in\mathcal D$,
    \begin{equation}\label{line513}
\mathbb P_W\{\sigma_{min}(T)\leq h/2\}
\leq N^{-D_*}+e^{-wN}\leq2N^{-D_*}.
    \end{equation}
With $C_1:=(h/2)\mathbf c^{(\alpha-2)/(2\alpha)}$, undoing the
normalization and averaging over labels yields
$$
\mathbb P\{\sigma_{min}(X)
\leq C_1n^{1/\alpha}(\log n)^{2(\alpha-2)/\alpha}\}
\leq e^{-wN}+2N^{-D_*}\leq3N^{-D_*}.
$$
Together with the upper bound, the failure probability is at most
$3N^{-D_*}+n^{-D_*}\leq4n^{-D-1}\leq n^{-D}$.
The constants $C_1,C_2$ depend only on $D,\delta,\alpha,C_\ell,C_u$;
$N_0$ may additionally depend on the tail threshold $T_0$.
\end{proof}

\subsection{Diverging aspect ratio}\label{divergingaspect}

Now we drop the assumption that $N/n$ is bounded. We sample $\Psi$ and $Y,Z$ slightly differently by using threshold depending on $N$ instead of $n$. More precisely, we now define
\begin{equation}\label{generationoflabelsunbound}
\psi_{ij}=\begin{cases}1,\quad\text{with probability } \mathbb{P}(|x_{ij}|< N^{\frac{1}{\alpha}-\tilde{\epsilon}_N}),\\0,\quad
\text{with probability } \mathbb{P}(|x_{ij}|\geq N^{\frac{1}{\alpha}-\tilde{\epsilon}_N}),
\end{cases}    \end{equation}
and for each interval $\tilde{I}\subset\mathbb{R}$, we define
\begin{equation}\label{resampling1}
\mathbb{P}(y_{ij}\in\tilde{I})=\frac{\mathbb{P}\left(x_{ij}\in\tilde{I}\cap(-N^{\frac{1}{\alpha}-\tilde{\epsilon}_N},N^{\frac{1}{\alpha}-\tilde{\epsilon}_N})\right)}{\mathbb{P}(x_{ij}\in(-N^{\frac{1}{\alpha}-\tilde{\epsilon}_N},N^{\frac{1}{\alpha}-\tilde{\epsilon}_N}))},\end{equation}
where $\tilde{\epsilon}_N$ is a cutoff value to be determined later, which is either a function of $N$, or a function of both $N$ and $n$. We define the law of $z_{ij}$ similarly, by swapping all factors of $n$ by $N$ in the definition \eqref{lawofz}.

\begin{proof}[\proofname\ of Theorem \ref{lowergeneral}]
Use \eqref{generationoflabelsunbound}--\eqref{resampling1}, with
$Y=(a_{ij}y_{ij})$ and $Z=(a_{ij}z_{ij})$, and choose
$N^{\alpha\tilde\epsilon_N}=\mathbf c(\log N)^4$. Set
$$
b_N:=N^{-1/\alpha+(1-\alpha/2)\tilde\epsilon_N},\qquad
T_1:=b_N\Psi\circ Y,\qquad
T_0:=b_N((\mathbf1-\Psi)\circ Z+B).
$$
Then $T:=T_0+T_1$ has the law of $b_N(A+B)$.
Lemma \ref{lemma2.34}, with $n$ replaced by $N$, gives
$s_{ij}^2:=N\operatorname{Var}(b_Na_{ij}y_{ij})\in[c_*,C_*]$
for sufficiently large $N$. These constants depend only on
$\alpha,C_\ell,C_u,A_1,A_2$, with $C_*$ independent of $A_1$.
Conditionally on $\Psi=W,Z$, the matrix $T_0$ is deterministic and
the entries of $T_1$ are independent and centered, with variances
$w_{ij}s_{ij}^2/N$ and absolute values at most
$A_2\mathbf c^{-1/2}(\log N)^{-2}$.

Let $G=T_0+G_1$, where $G_1$ has independent centered Gaussian entries
with these conditional variances and is independent of $Y,Z$ given
the label. With $r=N/n$,
apply Theorem \ref{derivation2.414} to $\sqrt r\,T$ and $\sqrt r\,G$:
the entry variances are at most $C_*/n$ and
$q=A_2\sqrt r\,\mathbf c^{-1/2}(\log N)^{-2}$.
Dividing the comparison error by $\sqrt r$ gives
$$
C\left[
N^{-1/2}t^{1/2}
+A_2^{1/3}\mathbf c^{-1/6}(\log N)^{-2/3}t^{2/3}
+A_2\mathbf c^{-1/2}(\log N)^{-2}t
\right],
$$
where $C$ depends only on $C_*$. Taking $t=(D+3)\log N$ gives, for
fixed $\mathbf c\geq1$ and sufficiently large $N$, uniformly over
labels after averaging over $Z$,
\begin{equation}\label{495proofofupperbound}
\mathbb{P}_{\Psi}\left\{
|\sigma_{min}(T)-\sigma_{min}(G)|
\geq C_D\mathbf c^{-1/6}\right\}
\leq8Ne^{-t}\leq N^{-(D+1)}.
\end{equation}
The constant $C_D$ depends only on $D,\alpha,C_\ell,C_u,A_2$, not on $r$.

To lower bound the Gaussian model, put $s_0:=\sqrt{c_*}$ and realize
$$
\sqrt N(G_1)_{ij}
=\psi_{ij}\left(s_0\xi_{ij}
+\sqrt{s_{ij}^2-s_0^2}\,\eta_{ij}\right).
$$
Here $\xi,\eta$ are independent standard Gaussian arrays, independent
of $\Psi,Y,Z$; both parts preserve the mask. Since
$\mathbb P(\psi_{ij}=1)\to1$, Theorem \ref{universalbacks} applied
to $\xi$ gives constants $h_{\rm G},w>0$ and a set of labels
$\mathcal D$ with $\mathbb P(\Psi\notin\mathcal D)\leq e^{-wN}$,
uniform in the deterministic part. For fixed $W\in\mathcal D,Z,\eta$,
the matrix $\sqrt N G/s_0$ is an arbitrary deterministic matrix plus
$(w_{ij}\xi_{ij})$. Applying the theorem and then averaging over
$Z,\eta$ yields, with $h:=s_0h_{\rm G}>0$,
$$
\mathbb{P}_W\{\sigma_{min}(G)\leq h\}\leq e^{-wN}
\qquad\text{for every }W\in\mathcal D.
$$

Choose $\mathbf c\geq1$ so that $C_D\mathbf c^{-1/6}\leq h/2$,
then $N_0\geq3$ so that all preceding estimates hold and
$e^{-wN}\leq N^{-(D+1)}$. As in the bounded-aspect-ratio proof,
\eqref{495proofofupperbound} and the bad-label probability give
$$
\mathbb{P}\{\sigma_{min}(T)\leq h/2\}
\leq N^{-(D+1)}+2e^{-wN}\leq N^{-D}.
$$
Finally, undo the normalization using
$$
b_N^{-1}
=\mathbf c^{(\alpha-2)/(2\alpha)}
N^{1/\alpha}(\log N)^{2(\alpha-2)/\alpha}.
$$
This gives $C_3=(h/2)\mathbf c^{(\alpha-2)/(2\alpha)}$, depending
only on $D,\delta,\alpha,C_\ell,C_u,A_1,A_2$; $N_0$ may additionally
depend on $T_0$. Neither depends on $B$.
\end{proof}

Finally we prove the upper bound $\sigma_{min}(X)$ in Proposition \ref{propositionunbounded} when $N/n$ is unbounded.

\begin{proof}[\proofname\ of Proposition \ref{propositionunbounded}]
Apply Lemma \ref{smallcolumnupper} with the prescribed
$\mathbf b\in(0,1)$, taking $C_4=C_{\rm up}$ and $N_0=N_{\rm up}$.
It gives the stronger failure bound $e^{-n^{\mathbf b}}$, which is
at most $N^{-D}+e^{-n^{\mathbf b}}$ for every $D>0$.
The lemma is uniform over $2\leq n\leq N$, so no upper bound on
$N/n$ is required.
\end{proof}

\subsection{The case of a general regularly varying tail}
We prove Theorem \ref{slowvaryingtails} by choosing the truncation
through the tail distribution, rather than assuming a lower bound on $L$.
\begin{proof}[\proofname\ of Theorem \ref{slowvaryingtails}]
Write $\bar F(t)=\mathbb{P}(|\zeta|\geq t)$ and fix $D>0$.
By asymptotic inversion \cite[Theorem 1.5.12]{bingham1987regular},
$1/\bar F$ has an asymptotic inverse $u$ regularly varying with
index $1/\alpha$. For $\mathbf c\geq1$ to be chosen below, set
$$
b_n:=u\!\left(\frac{n}{\mathbf c(\log n)^4}\right),
\qquad n\bar F(b_n)\sim\mathbf c(\log n)^4.
$$
Use the resampling decomposition \eqref{resamplingdecision} with
cutoff $b_n$, sampling all labels and entries of $Y,Z$ independently.
By \eqref{truncatedsecondmoment} and Karamata's theorem
\cite[Theorem 1.5.11]{bingham1987regular},
$$
v_n^2:=\mathbb{E}y_{ij}^2
\sim\frac{\alpha}{2-\alpha}\,b_n^2\bar F(b_n),
\qquad s_n:=\sqrt n\,v_n.
$$
Set
$$
T_1:=s_n^{-1}\Psi\circ Y,\qquad
T_0:=s_n^{-1}(\mathbf1-\Psi)\circ Z,\qquad T:=T_0+T_1.
$$
Thus $X\overset{\mathrm{law}}{=}s_nT$.
For each fixed label, symmetry makes the active entries of $T_1$
centered with variance $1/n$, and their absolute values are bounded by
$$
q_n:=\frac{b_n}{s_n}
\sim\sqrt{\frac{2-\alpha}{\alpha\mathbf c}}\,(\log n)^{-2}.
$$
Define $G:=T_0+n^{-1/2}(\psi_{ij}\gamma_{ij})$ using independent
standard real Gaussians also independent of $\Psi,Y,Z$.
Conditioning on the label $W$ and on $Z$ makes $T_0$ deterministic.
Theorem \ref{derivation2.414}, with $t=(D+3)\log N$, then bounds the
comparison error by $C_{D,\alpha,\delta}\mathbf c^{-1/6}+o(1)$
with failure probability at most $8Ne^{-t}$, uniformly in $W,Z$.

For $\delta_0=(1+\delta)/2$, eventually $N\geq\delta_0n$ and
the label success probability tends to one. Apply Theorem
\ref{universalbacks} to $\sqrt n\,G$ and average over $Z$ and the
labels. Its uniformity in the deterministic part gives
$$
\mathbb P\{\sigma_{min}(G)\leq h\}\leq2e^{-wN},
$$
where $h,w>0$ depend only on $\delta$.
Choose $\mathbf c$ first and then $n_0\geq3$ sufficiently large
so that the comparison error is at most $h/2$ and
$8Ne^{-t},e^{-wN}\leq n^{-D-1}$ for $n\geq n_0$.
Averaging the comparison bound and combining it with the last estimate
yields
$$
\mathbb{P}\{\sigma_{min}(X)\leq hs_n/2\}
\leq8Ne^{-t}+2e^{-wN}
\leq3n^{-D-1}\leq n^{-D}.
$$
For large real $x$, define
$$
b(x):=u\!\left(\frac{x}{\mathbf c(\log x)^4}\right),\qquad
s(x):=\sqrt{x\,\mathbb{E}[\zeta^2\mid|\zeta|<b(x)]}.
$$
The inverse and truncated-moment asymptotics show that both $b$ and $s$ are
regularly varying with index $1/\alpha$. Hence
$$
L_D'(x):=\frac{2x^{1/\alpha}}{h\,s(x)}
$$
is positive and slowly varying; extend it positively over the remaining
bounded interval. Since $s(n)=s_n$, this is the required function.
\end{proof}

\section{Proof of technical results}

This section collects the proof of Theorem \ref{derivation2.414} and \ref{universalbacks}. Section \ref{subsection3.1} towards Section \ref{section3.4theends} contain the proof of Theorem \ref{universalbacks} and Section \ref{section3.5} contains the proof of Theorem \ref{derivation2.414}.

\subsection{Anti-concentration: preliminary results}\label{subsection3.1}

As a preparation for the proof of Theorem \ref{universalbacks}, we recall several facts from \cite{tikhomirov2016smallest}. The first is about bounding the least singular value without an operator norm control.

For a subspace $E\subset\mathbb{R}^n$ we denote by $\operatorname{proj}_E$ the orthogonal projection of $\mathbb{R}^n$ onto $E$, and denote by $E^\perp$ the orthogonal complement of $E$ in $\mathbb{R}^n$.

\begin{Proposition}\label{proposition3.13.13.1}[\cite{tikhomirov2016smallest},Proposition 3]
    Let $D_1,D_2$ be $N\times n$ (deterministic) matrices and $D=D_1+D_2$. Fix a set of vectors $S\subset\mathbb{S}^{n-1}$. Assume that we can find $h,\epsilon>0$, some subset $\mathcal{N}\subset\mathbb{R}^n$, and linear subspaces $\{E_{y'}\subset\mathbb{R}^n:y'\in\mathcal{N}\}$ such that $y'\in E_{y'}$ for all $y'\in\mathcal{N}$, and that (1) For all $y'\in\mathcal{N}$, $$
\operatorname{dist}(D_1y',D(E_{y'}^\perp)+D_2(E_{y'}))\geq h; 
    $$

    (2) For each $y\in S$ we can find $y'\in\mathcal{N}$ with 
    $$
\|\operatorname{Proj}_{E_{y'}}(y)-y'\|_2\leq\epsilon.
    $$  Then we have 
    $$
\inf_{y\in S}\|Dy\|_2\geq h-\epsilon\|D_1\|.    
$$\end{Proposition}

We also need a few facts about anti-concentration and subspace projection. For a random variable $\xi$ we define its Lévy concentration function via 

\begin{equation}
\label{levyconcentrationf}\mathcal{Q}(\xi,\alpha)=\sup_{\lambda\in\mathbb{R}}\mathbb{P}(|\xi-\lambda|\leq\alpha).
\end{equation}

We will also need fixed-mask and conditional versions of the Lévy concentration function. For any fixed label $\Psi\in\{0,1\}^{N\times n}$ or any label event $\Omega_\Psi$ of positive probability, we define analogously
\begin{equation}\label{631analogously}
\mathcal{Q}_\Psi(\xi,\alpha)=\sup_{\lambda\in\mathbb{R}}\mathbb{P}_\Psi(|\xi-\lambda|\leq\alpha),\quad \mathcal{Q}_{\Omega_\Psi}(\xi,\alpha)=\sup_{\lambda\in\mathbb{R}}\mathbb{P}_{\Omega_\Psi}(|\xi-\lambda|\leq\alpha), 
\end{equation}where we recall \eqref{definition1w}, \eqref{definition2label} for the notation of $\mathbb{P}_\Psi$ and $\mathbb{P}_{\Omega_\Psi}.$

\begin{theorem}\label{theorem3.2}[\cite{MR131894}]
    Fix $k\in\mathbb{N}$ and consider $\xi_1,\cdots,\xi_k$ independent random variables. Fix real numbers $h_1,\cdots,h_k>0$. Then given any $h\geq \max_{j=1,\cdots,k}h_j$,
    $$
\mathcal{Q}(\sum_{j=1}^k\xi_j,h)\leq C_{\eqref{theorem3.2}}h\left(\sum_{j=1}^k(1-\mathcal{Q}(\xi_j,h_j))h_j^2\right)^{-1/2}
    $$
    for a universal constant $C_{\eqref{theorem3.2}}>0$.
\end{theorem}

Then we have a convenient bound on anti-concentration of subspace projection
\begin{corollary}\label{corollary3.3}[\cite{tikhomirov2016smallest}, Corollary 6]
    Consider $X=(X_1,\cdots,X_m)$ a vector with independent coordinates with $\mathcal{Q}(X_i,h)\leq 1-\tau$ for some $h>0,\tau\in(0,1)$ and for each $i=1,\cdots,m$. For any $\ell\in\mathbb{N}$, and any $d$-dimensional fixed subspace ($d\leq m$) $E\subset\mathbb{R}^m$, we have 
    $$
\mathcal{Q}(\operatorname{Proj}_EX,h\sqrt{d}/\ell)\leq (C_{\eqref{corollary3.3}}/\sqrt{\ell\tau})^{d/\ell},
    $$ where $C_{\eqref{corollary3.3}}$ is some universal constant.
\end{corollary}

We recall a well-known operator norm bound for matrices with sub-Gaussian entries; see, for example, \cite{rudelson2010non} or \cite{benhamou2018operator}. We include a short proof to make explicit that a fixed label and non-unit entry variance are allowed.
\begin{lemma}\label{lemma3.444}
    Fix a label $\Psi=(\psi_{ij})\in\{0,1\}^{N\times n}$, and let $W=(\psi_{ij}w_{ij})$ be an $N\times n$ random matrix ($N\geq n$) where $w_{ij}$ are i.i.d. mean-zero $K$-sub-Gaussian random variables. No normalization of their variance is required. Then we can find $C_{\eqref{lemma3.444}}>0$ depending only on $K$ (and independent of $\Psi$) such that 
    $$
\mathbb{P}\{\|W\|\geq C_{\eqref{lemma3.444}}\sqrt{N}\}\leq\exp(-N).
    $$
\end{lemma}

\begin{proof}
For fixed unit vectors $u\in\mathbb{R}^N$ and $v\in\mathbb{R}^n$, independence, centering, and the sub-Gaussian moment generating function bound give
$$
\mathbb{P}\{|u^{\mathsf T}Wv|\geq s\}
\leq 2\exp(-cs^2/K^2),\qquad s>0,
$$
where $c>0$ is universal, since
$\sum_{i,j}(u_i\psi_{ij}v_j)^2\leq1$.
Choose $1/4$-nets of the two unit spheres with cardinalities at most
$9^N$ and $9^n$. The standard net approximation gives
$\|W\|\leq2\max_{u,v}|u^{\mathsf T}Wv|$, where the maximum is over these nets.
Thus
$$
\mathbb{P}\{\|W\|\geq2s\}
\leq2\cdot9^{N+n}\exp(-cs^2/K^2).
$$
Taking $s=C K\sqrt{N}$ with a sufficiently large universal $C$, and using $n\leq N$, proves the claim uniformly over the label.
\end{proof}

We finally quote a purely technical result from \cite{tikhomirov2016smallest} which gives improved cardinality counting of the net.
This lemma is possibly trivial when $\xi$ is a standard Gaussian, but let us state it that way.
\begin{lemma}\label{lemma3.56}(\cite{tikhomirov2016smallest}, Lemma 14)
    Consider $\xi$ a random variable satisfying, for $z\in\mathbb{R},\gamma>0,N\in\mathbb{N}$:
    $$ \min(\mathbb{P}(z-\sqrt{N}\leq\xi\leq z-1),\mathbb{P}(z+1\leq\xi\leq z+\sqrt{N}))\geq\gamma.
    $$ Then we can find some integer $\ell\in[0,\lfloor\log_2\sqrt{N}\rfloor]$, some $\lambda\in\mathbb{R}$, some Borel subsets $H_1,H_2\subset[-2^{\ell+2},2^{\ell+2}]$ satisfying $$\operatorname{dist}(H_1,H_2)\geq 2^\ell, \quad \min(\mathbb{P}\{\xi-\lambda\in H_1\},\mathbb{P}\{\xi-\lambda\in H_2\})\geq c_{\eqref{lemma3.56}}\gamma 2^{-\ell/8}$$ where $c_{\eqref{lemma3.56}}>0$ is a universal constant, and such that for $H=H_1\cup H_2$ we have $$\mathbb{E}[(\xi-\lambda)1_{\xi-\lambda\in H}]=0.$$
\end{lemma}

\subsection{Sphere decomposition and cardinality of nets}

We now outline a decomposition of the sphere $\mathbb{S}^{n-1}$ into three subsets, following \cite{tikhomirov2016smallest}. \begin{Definition} For fixed $\theta>0,m>0$ we define subsets $\mathbb{S}_p^{n-1}(\theta)$ and $\mathbb{S}_a^{n-1}(m)$ of $\mathbb{S}^{n-1}$ via
    \begin{enumerate}
\item The set $\mathbb{S}_p^{n-1}(\theta)$ of $\theta$-peaky vectors consist of unit vectors in $\mathbb{R}^n$ that have $\ell_\infty^n$ norm at least $\theta$. \item  A vector $y\in \mathbb{S}^{n-1}$ is called $m$-sparse if $|\operatorname{supp}y|\leq m$. We then call $y\in\mathbb{S}^{n-1}$ almost $m$-sparse if we can find $J\subset\{1,\cdots,n\}$ with cardinality bounded by $m$, and $\|y\chi_J\|_2\geq\frac{1}{2}$. We use the notation $\mathbb{S}_a^{n-1}(m)$ to denote the set of almost $m$-sparse vectors of $\mathbb{S}^{n-1}$.
\end{enumerate}\end{Definition}
The cardinality of nets for $\mathbb{S}_a^{n-1}$ can be bounded by the following lemma.

\begin{lemma}\label{lemma3.7}[\cite{tikhomirov2016smallest}, Lemma 12] We can find a universal constant $C_{\eqref{lemma3.7}}>0$ so the following holds. Given $n,m\in\mathbb{N},n\geq m$, $\epsilon\in(0,1]$, a subset $S\subset\mathbb{S}^{n-1}$, and let $T\subset\mathbb{B}_2^n$ (the unit ball in $\ell_2^n$) be a subset of $m$-sparse vectors that satisfy 
\begin{equation}\label{line444}
\text{for each }y\in S\text{ we can find }x=x(y)\in T\text{ such that }y\chi_{\text{supp}x}=x.
\end{equation} Then we can find a set $\mathcal{N}\subset T$ with cardinality bounded by $(\frac{C_{\eqref{lemma3.7}}n}{\epsilon m})^m$ so that given any $y\in S$ we can find $y'=y'(y)\in\mathcal{N}$ satisfying $\|y'-\chi_{\operatorname{supp}y'}y\|_2\leq\epsilon.$
    
\end{lemma}

Finally we consider the rest of the sphere $\mathbb{S}^{n-1}\setminus (\mathbb{S}_p^{n-1}(\theta)\cup\mathbb{S}_a^{n-1}(m))$, and find a net.

\begin{lemma}\label{lemma3.82}(\cite{tikhomirov2016smallest}, Lemma 16)
    Fix some $N\geq n\geq m\geq 1$. For any $y\in\mathbb{S}^{n-1}\setminus \mathbb{S}_a^{n-1}(\sqrt{N})$ we can find s subset $J=J(y)\subset\{1,2,\cdots,n\}$ with $|J|\leq m$, $\|y\chi_J\|_2\geq\frac{1}{2}\sqrt{\frac{m}{n}}$, and that $\|y\chi_J\|_\infty\leq\frac{1}{\lfloor N^{1/4}\rfloor}.$
\end{lemma}

\subsection{Quasi-random selection of labels}

Throughout this subsection, we use the decomposition $G=G_0+G_1$,
where $G_0$ is deterministic and $G_1=(\psi_{ij}g_{ij})$.
Unless otherwise specified, the $g_{ij}$ are i.i.d. mean-zero
$K$-sub-Gaussian random variables, and the labels $\psi_{ij}$ are i.i.d.
Bernoulli variables independent of the $g_{ij}$, with
$\mathbb{P}(\psi_{ij}=1)\geq(3\delta+1)/(4\delta)$ for the fixed $\delta>1$.
Unlike Definition \ref{line276definition}, the auxiliary results here do
not require unit entry variance: the operator norm bound follows from
Lemma \ref{lemma3.444}, and the remaining distributional assumptions are
stated in the individual results. In particular, Lemmas \ref{lemma3.9}
and \ref{lemma3.10} hold under the weaker assumptions given below.

We first prove that quasi-random selection of labels holds for a single vector.

\begin{lemma}\label{lemma3.9}
    Fix $\delta>1$ and $d,r>0$. Let $(g_{i1},g_{i2},\cdots,g_{in})$ be a random vector of i.i.d. coordinates, and let $(g_{i1}',g_{i2}',\cdots,g_{in}')$ be an independent copy. Let $(\psi_{i1},\psi_{i2},\cdots,\psi_{in})$ be independent Bernoulli random variables, independent of both vectors, such that $\mathbb{P}(\psi_{ij}=1)\geq\frac{1}{2}$ for every $j$. Consider $H\subset\mathbb{R}$ a Borel subset with $H=H_1\cup H_2$ for two Borel sets $H_1,H_2$, where $\operatorname{dist}(H_1,H_2)\geq d$, and $\min(\mathbb{P}(g_{ij}\in H_1),\mathbb{P}(g_{ij}\in H_2))\geq r$. Given any $t>0$ we define 
    $$
h_{\eqref{lemma3.9}}=\frac{1-\delta^{-1/4}}{C_{\eqref{theorem3.2}}}\frac{r}{4}td,
    $$ and let $y\in\mathbb{R}^n$ be a vector which satisfy $\|y\|_2\geq t$, $\|y\|_\infty\leq \frac{2h_{\eqref{lemma3.9}}}{d}$. Then 
    \begin{equation}\label{andletsumover}
\mathbb{P}\left\{\left|\sum_{j=1}^n \psi_{ij}(g_{ij}-g_{ij}')y_j\right|\leq h_{\eqref{lemma3.9}}\right\}
  \leq 1-\delta^{-1/4}.\end{equation}
\end{lemma}
We note that the randomness is taken jointly over $g_{ij},g_{ij}'$ and $\psi_{ij}$.
\begin{proof}
    From joint independence and the definition of $\psi_{ij}$, we have
    $$
\begin{aligned}
\mathbb{P}\{\psi_{ij}|g_{ij}-g_{ij}'|\geq d\}
&\geq \mathbb{P}\{g_{ij}\in H_1,g_{ij}'\in H_2,\psi_{ij}=1\}
+\mathbb{P}\{g_{ij}\in H_2,g_{ij}'\in H_1,\psi_{ij}=1\}\\
&=2\mathbb{P}(\psi_{ij}=1)\mathbb{P}(g_{ij}\in H_1)\mathbb{P}(g_{ij}\in H_2)\geq r^2,
\end{aligned}
    $$
    where we use $\mathbb{P}(\psi_{ij}=1)\geq\frac{1}{2}$. The random variable $\psi_{ij}(g_{ij}-g_{ij}')$ has a symmetric distribution, so
    $$
    \mathcal{Q}\left(\psi_{ij}(g_{ij}-g_{ij}'),\frac{d}{2}\right)\leq 1-\frac{r^2}{2}\leq 1-\frac{r^2}{4}.
    $$
    By our assumption, $h_{\eqref{lemma3.9}}\geq \frac{d|y_j|}{2}$ for each $j$. Applying Theorem \ref{theorem3.2} to the nonzero coordinates of $y$, with $h_j=|y_j|d/2$, gives
$$\begin{aligned}
\text{LHS of } \eqref{andletsumover}
&\leq C_{\eqref{theorem3.2}}h_{\eqref{lemma3.9}}
\left(\frac{1}{4}\sum_{j\in\operatorname{supp}y}
\left(1-\mathcal{Q}\left(\psi_{ij}(g_{ij}-g_{ij}')y_j,\frac{|y_j|d}{2}\right)\right)(y_jd)^2\right)^{-1/2}\\
&\leq C_{\eqref{theorem3.2}} h_{\eqref{lemma3.9}}
\left(\frac{r^2}{16}\sum_{j=1}^n(y_jd)^2\right)^{-1/2}
\leq 1-\delta^{-1/4},
\end{aligned}$$
    where the last inequality uses $\|y\|_2\geq t$.
    This concludes the proof.
\end{proof}

Now we can conclude the quasi-random label selection for a fixed vector $y$ with sufficient mass and which is not peaky.

\begin{lemma}\label{lemma3.10}
Fix $\delta>1$, $N\geq\delta n$, and $t,d,r>0$.
Let $(g_{ij})_{i\in[N],j\in[n]}$ be i.i.d. real random variables.
Let $\Psi=(\psi_{ij})$ consist of mutually independent Bernoulli
variables, independent of all the $g_{ij}$, such that
$\mathbb{P}(\psi_{ij}=1)\geq1/2$ for every $i,j$.
Suppose that $H_1,H_2\subset\mathbb{R}$ are Borel sets satisfying
$$
\operatorname{dist}(H_1,H_2)\geq d,
\qquad
\min_{s\in\{1,2\}}\mathbb{P}(g_{ij}\in H_s)\geq r.
$$
Set $G_1=(\psi_{ij}g_{ij})$ and, as in Lemma \ref{lemma3.9}, define
$$
h_{\eqref{lemma3.9}}
:=\frac{1-\delta^{-1/4}}{C_{\eqref{theorem3.2}}}\frac{r}{4}td.
$$
For any $y\in\mathbb{R}^n$ with $\|y\|_2\geq t$ and
$\|y\|_\infty\leq2h_{\eqref{lemma3.9}}/d$, we can find a subset of
labels $\Omega_y\subset\{0,1\}^{N\times n}$ such that
    $$
\mathbb{P}(\Psi\in\Omega_y)\geq 1-\exp(-w_{\eqref{lemma3.10}}N)
    $$ and that
    for any $W_y\in\Omega_y$ and any deterministic $G_0\in\mathbb{R}^{N\times n}$,
    $$
\mathbb{P}_{W_y}\{\operatorname{dist}(G_1y,V_{G_0,G_1}(E))\leq h_{\eqref{lemma3.10}}h_{\eqref{lemma3.9}}\sqrt{N}\}\leq 2\exp(-w_{\eqref{lemma3.10}}N)
    $$ for constants $h_{\eqref{lemma3.10}},w_{\eqref{lemma3.10}}$ depending only on $\delta$. Here $E:=\operatorname{span}\{e_j\}_{j\in\operatorname{supp}y},$ and for any subspace $E\subset\mathbb{R}^n$, we denote by $V_{G_0,G_1}(E)$ the following subspace
    \begin{equation}\label{vgog1e}V_{G_0,G_1}(E):=(G_0+G_1)(E^\perp)+G_0(E).\end{equation}
\end{lemma}

\begin{proof}

Let $(g_{ij}')$ be an independent copy of $(g_{ij})$, independent also
of the labels. For $i\in[N]$ and $J\subset[n]$, write
$$
S_i(J):=\sum_{j\in J}g_{ij}y_j,\qquad
S_i'(J):=\sum_{j\in J}g_{ij}'y_j,\qquad
p_i(J):=\mathbb{P}\{\{j:\psi_{ij}=1\}=J\}.
$$
These sums are defined for every row pattern $J$, including those with
$p_i(J)=0$, so no conditioning on a null event is needed.
Set
$$
\tau:=\frac{\delta^{-1/4}-\delta^{-1/3}}{2}.
$$
If a pattern $J$ satisfies
\begin{equation}\label{assumptiononthej}
\mathbb{P}\{|S_i(J)-S_i'(J)|>h_{\eqref{lemma3.9}}\}\geq2\tau,
\end{equation}
then, by independence of the two copies, for every $\lambda\in\mathbb{R}$,
$$
\begin{aligned}
\mathbb{P}\{|S_i(J)-\lambda|\leq h_{\eqref{lemma3.9}}/2\}^2
&\leq\mathbb{P}\{|S_i(J)-S_i'(J)|\leq h_{\eqref{lemma3.9}}\}\\
&\leq1-2\tau.
\end{aligned}
$$
Consequently,
$\mathcal{Q}(S_i(J),h_{\eqref{lemma3.9}}/2)\leq\sqrt{1-2\tau}\leq1-\tau$.
Define
$$
L_i:=\{J\subset[n]:\mathcal{Q}(S_i(J),h_{\eqref{lemma3.9}}/2)\leq1-\tau\},
\qquad
\mathcal{E}_i:=\{\{j:\psi_{ij}=1\}\in L_i\}.
$$
Every pattern satisfying \eqref{assumptiononthej} belongs to $L_i$.
Using independence of the labels and entries, Lemma \ref{lemma3.9}
therefore gives
$$
\begin{aligned}
\delta^{-1/4}
&\leq\sum_{J\subset[n]}p_i(J)
\mathbb{P}\{|S_i(J)-S_i'(J)|>h_{\eqref{lemma3.9}}\}\\
&\leq\sum_{J\in L_i}p_i(J)+2\tau\sum_{J\notin L_i}p_i(J)
\leq\mathbb{P}(\mathcal{E}_i)+2\tau.
\end{aligned}
$$
Thus $\mathbb{P}(\mathcal{E}_i)\geq\delta^{-1/3}$.
The events $\mathcal{E}_i$ are independent. For a fixed mask $W=(w_{ij})$,
let
$$
I(W):=\{i\in[N]:\{j:w_{ij}=1\}\in L_i\},\qquad
\Omega_y:=\{W\in\{0,1\}^{N\times n}:|I(W)|\geq N\delta^{-1/2}\}.
$$
Bernstein's inequality yields a constant $c_\delta>0$ such that
$$
\mathbb{P}(\Psi\in\Omega_y)\geq1-\exp(-c_\delta N).
$$

Fix $W\in\Omega_y$, put $I=I(W)$ and $m=|I|$, and let
$P_I:\mathbb{R}^N\to\mathbb{R}^m$ be the coordinate restriction to $I$.
Under $\mathbb{P}_W$, the vector $X:=P_I G_1y$ has independent
coordinates, each with concentration function at scale
$h_{\eqref{lemma3.9}}/2$ bounded by $1-\tau$.
Set
$$
\kappa:=\delta^{-1/2}-\delta^{-1},\qquad
k:=m-n\geq\kappa N>0,\qquad
\ell:=\left\lceil\max\left\{1,\frac{4C_{\eqref{corollary3.3}}^2}{\tau}\right\}\right\rceil.
$$
For any fixed subspace $F\subset\mathbb{R}^N$ with $\dim F\leq n$,
we have $\dim P_I F\leq n$. Choose a $k$-dimensional subspace
$L\subset(P_I F)^\perp$. Then
$$
\operatorname{dist}(G_1y,F)
\geq\operatorname{dist}(X,P_I F)
\geq\|\operatorname{proj}_L X\|_2.
$$
Corollary \ref{corollary3.3} now implies
$$
\begin{aligned}
\mathbb{P}_W\left\{\operatorname{dist}(G_1y,F)
\leq\frac{h_{\eqref{lemma3.9}}\sqrt{\kappa N}}{2\ell}\right\}
&\leq\mathbb{P}_W\left\{\|\operatorname{proj}_L X\|_2
\leq\frac{h_{\eqref{lemma3.9}}\sqrt{k}}{2\ell}\right\}\\
&\leq\left(\frac{C_{\eqref{corollary3.3}}}{\sqrt{\ell\tau}}\right)^{k/\ell}
\leq2^{-\kappa N/\ell}.
\end{aligned}
$$
For $E=\operatorname{span}\{e_j:j\in\operatorname{supp}y\}$, the subspace
$V_{G_0,G_1}(E)$ has dimension at most $n$. Under the fixed-mask law,
it depends only on columns outside $\operatorname{supp}y$, whereas
$G_1y$ depends only on columns in that support. Conditioning on the
former columns makes $V_{G_0,G_1}(E)$ fixed without changing the law of
$G_1y$. Applying the preceding bound and then removing the conditioning
proves the lemma with
$$
h_{\eqref{lemma3.10}}:=\frac{\sqrt{\kappa}}{2\ell},\qquad
w_{\eqref{lemma3.10}}:=\min\{c_\delta,\kappa\log2/\ell\}.
$$
\end{proof}

Now we can take the union bound over almost sparse vectors.

\begin{Proposition}\label{proposition3.11}(Almost sparse vectors from $\mathbb{S}_a^{n-1}(\sqrt{N})$).
Fix $\gamma>0$ and $\delta>1$. We can find
$N_\eqref{proposition3.11}$ depending only on $\gamma,\delta,K$ and
$h_\eqref{proposition3.11}>0$ depending only on $\gamma,\delta$.
We define $\theta_\eqref{proposition3.11}=\frac{1-\delta^{-1/4}}{C_\eqref{theorem3.2}}\frac{\gamma}{4}$
    and let  $N\geq\max(N_\eqref{proposition3.11},\delta n)$. Suppose that $g_{ij}$ satisfies, for some $z\in\mathbb{R}$, that 
    $$
\min(\mathbb{P}\{z-\sqrt{N}\leq g_{ij}\leq z-1\},\mathbb{P}\{z+1\leq g_{ij}\leq z+\sqrt{N}\})\geq \gamma.
    $$ Then denote by $S:=\mathbb{S}_a^{n-1}(\sqrt{N})\setminus \mathbb{S}_p^{n-1}(\theta_{\eqref{proposition3.11}})$, we can find a collection of labels $\Omega_a\subset\{0,1\}^{N\times n}$, with $\mathbb{P}(\Psi\in\Omega_a)\geq 1-\exp(-w_\eqref{lemma3.10}N/2)$, such that for any $W_a\in \Omega_a$,
    $$
\mathbb{P}_{W_a}\{\inf_{y\in S}\|G_0y+G_1y\|_2\leq h_\eqref{proposition3.11}\sqrt{N}
\}\leq\exp(-w_\eqref{lemma3.10}N/2)+\exp(-N).
    $$
\end{Proposition}

\begin{proof}
For any fixed $\gamma>0$, $\delta>1$ we take $d=2$, $r=\gamma$ and $t=\frac{1}{2}$. We take $C_{\eqref{lemma3.7}}\geq1$, enlarging this universal constant if necessary. Let
$N_\eqref{proposition3.11}=N_\eqref{proposition3.11}(\gamma,\delta,K)$
be the smallest integer larger than $\frac{2\max(1,C_\eqref{lemma3.444})}{h_\eqref{lemma3.9}h_\eqref{lemma3.10}}$ such that 
   
    $$ 2(C_\eqref{lemma3.7}N)^{3\sqrt{N}}\leq\exp(w_\eqref{lemma3.10}N/2).
    $$ Take $H_1=[z-\sqrt{N},z-1]$, $H_2=[z+1,z+\sqrt{N}]$, and $H=H_1\cup H_2$. The hypothesis gives $\min(\mathbb{P}\{g_{ij}\in H_1\},\mathbb{P}\{g_{ij}\in H_2\})\geq\gamma$, while $\operatorname{dist}(H_1,H_2)=2$.

Set
$$
m_a:=\min\{n,\lfloor\sqrt{N}\rfloor\},\qquad \epsilon:=N^{-2}.
$$
Then $m_a$ is an integer with $1\leq m_a\leq n$ and
$0<\epsilon\leq1$, even when $\sqrt{N}>n$.
Let $T\subset B_2^n$ consist of the $m_a$-sparse vectors with
$\ell^2$ norm at least $1/2$ and $\ell^\infty$ norm at most
$\theta_{\eqref{proposition3.11}}$.
Since support sizes are integers and cannot exceed $n$, this is
exactly the set of $\sqrt{N}$-sparse vectors with the same norm
constraints. By the definition of $S$, the sets $S,T$ satisfy
\eqref{line444}. Thus Lemma \ref{lemma3.7} gives a finite net
$\mathcal{N}\subset T$ such that, for every $y\in S$, there is
$y'=y'(y)\in\mathcal{N}$ with
$\|y\chi_{\operatorname{supp}y'}-y'\|_2\leq N^{-2}$, and
$$
\begin{aligned}
|\mathcal{N}|
&\leq\left(\frac{C_{\eqref{lemma3.7}}nN^2}{m_a}\right)^{m_a}
\leq\left(C_{\eqref{lemma3.7}}N^3\right)^{m_a}\\
&\leq\left(C_{\eqref{lemma3.7}}N\right)^{3\sqrt{N}}.
\end{aligned}
$$
Here we used $n\leq N$, $m_a\geq1$, $m_a\leq\sqrt{N}$, and
$C_{\eqref{lemma3.7}}\geq1$.
    Let 
    $$
\Omega_a:=\cap_{y\in\mathcal{N}}\Omega_y, 
    $$where $\Omega_y$ was defined in Lemma \ref{lemma3.10}.
    Then combining Lemma \ref{lemma3.10} with a simple union bound, we have $\mathbb{P}(\Psi\in\Omega_a)\geq 1-\exp(-w_{\eqref{lemma3.10}}N/2)$. 

For any $W_a\in\Omega_a$, set
$$
H:=h_{\eqref{lemma3.9}}h_{\eqref{lemma3.10}}\sqrt{N},
$$
and consider the event
$$
\begin{aligned}
\mathcal{E}_{W_a}:=\{\omega\in\Omega:\;&\Psi(\omega)=W_a,
\ \|G_1(\omega)\|\leq C_{\eqref{lemma3.444}}\sqrt{N},\\
&\operatorname{dist}\bigl(G_1(\omega)y',
(G_0+G_1(\omega))(E_{y'}^\perp)+G_0(E_{y'})\bigr)\geq H
\quad\text{for every }y'\in\mathcal{N}\}.
\end{aligned}
$$
Here $E_{y'}=\operatorname{span}\{e_j\}_{j\in\operatorname{supp}y'}$. By Lemmas \ref{lemma3.10} and \ref{lemma3.444}, together with a union bound,
$$
\begin{aligned}
\mathbb{P}_{W_a}(\mathcal{E}_{W_a})
&\geq 1-2|\mathcal{N}|\exp(-w_{\eqref{lemma3.10}}N)-\exp(-N)\\
&\geq 1-\exp(-w_{\eqref{lemma3.10}}N/2)-\exp(-N).
\end{aligned}
$$
For every $\omega\in\mathcal{E}_{W_a}$, Proposition \ref{proposition3.13.13.1}, applied with $D_1=G_1(\omega)$, $D_2=G_0$, and $\epsilon=N^{-2}$, gives
$$
\begin{aligned}
\inf_{y\in S}\|(G_0+G_1(\omega))y\|_2
&\geq H-N^{-2}\|G_1(\omega)\|\\
&\geq h_{\eqref{lemma3.9}}h_{\eqref{lemma3.10}}\sqrt{N}
-C_{\eqref{lemma3.444}}N^{-3/2}\\
&\geq \frac{1}{2}h_{\eqref{lemma3.9}}h_{\eqref{lemma3.10}}\sqrt{N}.
\end{aligned}
$$
Thus, defining
$$
h_{\eqref{proposition3.11}}:=\frac{1}{4}h_{\eqref{lemma3.9}}h_{\eqref{lemma3.10}},
$$
the preceding lower bound on $\mathcal{E}_{W_a}$ is
$2h_{\eqref{proposition3.11}}\sqrt{N}>h_{\eqref{proposition3.11}}\sqrt{N}$.
Therefore,
$$
\mathbb{P}_{W_a}\left\{\inf_{y\in S}\|(G_0+G_1)y\|_2
\leq h_{\eqref{proposition3.11}}\sqrt{N}\right\}
\leq \exp(-w_{\eqref{lemma3.10}}N/2)+\exp(-N).
$$
\end{proof}

Next we take a union bound over the set of peaky vectors (vectors that have a large $\ell^\infty$ norm).
\begin{lemma}\label{lemma3.12}
    Fix $\delta>1$, $\gamma>0$ and let $N,n$ be that $N\geq \delta n$. Assume that $\mathcal{Q}(g_{ij},1)\leq 1-\gamma$. Then for any fixed $\theta>0$ there exists a collection of labels $\Omega_p\subset\{0,1\}^{N\times n},$ with $\mathbb{P}(\Psi\in\Omega_p)\geq 1-\exp(-w_\eqref{lemma3.12}N)$ such that for any $W_p\in\Omega_p,$
    $$
\mathbb{P}_{W_p}\{\inf_{y\in \mathbb{S}_p^{n-1}(\theta)}\|Gy\|_2\leq h_\eqref{lemma3.12}\theta\sqrt{N}\}\leq n\exp(-w_\eqref{lemma3.12}N),
    $$ where $h_\eqref{lemma3.12}$ and $w_\eqref{lemma3.12}$ depend only on $\delta$ and $\gamma$.
\end{lemma}

\begin{proof}
Set
$$
d:=N-n+1,\qquad
k:=\left\lceil\frac{d}{2}\right\rceil,\qquad
\rho_\delta:=\frac{\delta-1}{2\delta}.
$$
Since $N\geq\delta n$, we have
$$
k\geq\frac{d}{2}\geq\rho_\delta N.
$$
Let
$$
\ell_0:=\left\lceil\max\left\{1,\frac{4C_{\eqref{corollary3.3}}^2}{\gamma}\right\}\right\rceil,
\qquad
h_0:=\frac{\sqrt{\rho_\delta}}{\ell_0},
\qquad
w_0:=\frac{\rho_\delta\log 2}{\ell_0}.
$$
Define
$$
\Omega_p:=\left\{W=(w_{ij})\in\{0,1\}^{N\times n}:
\left|\{i\in[N]:w_{ij}=0\}\right|\leq\left\lfloor\frac{d}{2}\right\rfloor
\text{ for every }j\in[n]\right\}.
$$
Under the distribution of the labels,
$$
\mathbb{P}(\psi_{ij}=0)\leq\frac{\delta-1}{4\delta}=\frac{\rho_\delta}{2}.
$$
Thus, a Chernoff bound followed by a union bound gives
$$
\mathbb{P}\{\Psi\notin\Omega_p\}\leq n\exp(-c_\delta N)
\leq N\exp(-c_\delta N)
$$
for some $c_\delta>0$. For all sufficiently large $N$, the right-hand
side is at most $\exp(-c_\delta N/2)$. The all-one label belongs to
$\Omega_p$ and has probability bounded below by a positive constant,
depending only on $\delta$, over the finitely many remaining admissible
pairs $(N,n)$. Decreasing the constant to cover these pairs, we obtain
$\widetilde{w}_\delta>0$, depending only on $\delta$, such that
$$
\mathbb{P}\{\Psi\in\Omega_p\}\geq1-\exp(-\widetilde{w}_\delta N).
$$

Fix $W_p=(w_{ij})\in\Omega_p$ and $j\in[n]$, and condition on all columns of $G$ except the $j$-th one. Then
$$
F:=\operatorname{span}\{\operatorname{col}_s(G):s\neq j\}
$$
is a fixed subspace of dimension at most $n-1$, independent of the random variables in the $j$-th column. Put
$$
Z_j:=\{i\in[N]:w_{ij}=0\},
\qquad
F_0:=F+\operatorname{span}\{e_i:i\in Z_j\}.
$$
Since $|Z_j|\leq\lfloor d/2\rfloor$,
$$
\dim F_0\leq n-1+\left\lfloor\frac{d}{2}\right\rfloor=N-k.
$$
We may therefore choose a $k$-dimensional subspace $E\subset F_0^\perp$.

Consider the complete random column
$$
\mathbf{v}:=\bigl((G_0)_{1j}+g_{1j},\ldots,(G_0)_{Nj}+g_{Nj}\bigr)^{\mathsf{T}}.
$$
Its coordinates are independent and satisfy
$$
\mathcal{Q}(\mathbf{v}_i,1)=\mathcal{Q}(g_{ij},1)\leq1-\gamma.
$$
Moreover,
$$
\operatorname{col}_j(G)=\mathbf{v}-\sum_{i\in Z_j}g_{ij}e_i,
$$
so $\operatorname{col}_j(G)-\mathbf{v}\in F_0$. Hence
$$
\begin{aligned}
\operatorname{dist}(\operatorname{col}_j(G),F)
&\geq\operatorname{dist}(\operatorname{col}_j(G),F_0)\\
&=\|\operatorname{proj}_{F_0^\perp}\mathbf{v}\|_2
\geq\|\operatorname{proj}_{E}\mathbf{v}\|_2.
\end{aligned}
$$
Applying Corollary \ref{corollary3.3}, conditionally on the other columns, gives
$$
\begin{aligned}
\mathbb{P}_{W_p}\left\{\|\operatorname{proj}_{E}\mathbf{v}\|_2
\leq\frac{\sqrt{k}}{\ell_0}\,\middle|\,F\right\}
&\leq\left(\frac{C_{\eqref{corollary3.3}}}{\sqrt{\ell_0\gamma}}\right)^{k/\ell_0}\\
&\leq2^{-k/\ell_0}\leq\exp(-w_0N).
\end{aligned}
$$
Since
$$
h_0\sqrt{N}=\frac{\sqrt{\rho_\delta N}}{\ell_0}\leq\frac{\sqrt{k}}{\ell_0},
$$
removing the conditioning yields
$$
\mathbb{P}_{W_p}\left\{\operatorname{dist}\left(
\operatorname{col}_j(G),
\operatorname{span}\{\operatorname{col}_s(G):s\neq j\}\right)
\leq h_0\sqrt{N}\right\}\leq\exp(-w_0N).
$$

If $y\in\mathbb{S}_p^{n-1}(\theta)$, choose $j$ such that $|y_j|\geq\theta$. Then
$$
\begin{aligned}
\|Gy\|_2
&\geq |y_j|\operatorname{dist}\left(
\operatorname{col}_j(G),
\operatorname{span}\{\operatorname{col}_s(G):s\neq j\}\right)\\
&\geq\theta\operatorname{dist}\left(
\operatorname{col}_j(G),
\operatorname{span}\{\operatorname{col}_s(G):s\neq j\}\right).
\end{aligned}
$$
A union bound over $j\in[n]$ therefore gives
$$
\mathbb{P}_{W_p}\left\{\inf_{y\in\mathbb{S}_p^{n-1}(\theta)}\|Gy\|_2
\leq h_0\theta\sqrt{N}\right\}\leq n\exp(-w_0N).
$$
The conclusion follows upon setting
$$
h_{\eqref{lemma3.12}}:=h_0,
\qquad
w_{\eqref{lemma3.12}}:=\min\{w_0,\widetilde{w}_\delta\}.
$$
\end{proof}

Finally we consider the complement $\mathbb{S}^{n-1}\setminus(\mathbb{S}_a^{n-1}(\sqrt{N})\cup \mathbb{S}_p(\theta))$. The next Proposition and its proof is a reformulation of \cite{tikhomirov2016smallest}, Proposition 17.

\begin{Proposition}
    \label{propos13.13}
    Fix $\delta>1,\gamma>0$, and let $z\in\mathbb{R}$ be such that
    $$\min(\mathbb{P}\{z-\sqrt{N}\leq g_{ij}\leq z-1\},\mathbb{P}\{z+1\leq g_{ij}\leq z+\sqrt{N}\})\geq\gamma.
    $$ Denote by $S:=\mathbb{S}^{n-1}\setminus \mathbb{S}_a^{n-1}(\sqrt{N})$.
    Then we can find $N_\eqref{propos13.13}\in\mathbb{N}$ and
$h_\eqref{propos13.13}>0$, both depending only on $\gamma,\delta,K$,
such that, for $N\geq\max(N_\eqref{propos13.13},\delta n)$, we can find a subset of labels $\Omega_g\subset\{0,1\}^{N\times n}$ with $\mathbb{P}(\Psi\in\Omega_g)\geq 1-\exp(-w_\eqref{lemma3.10}N/2)$, such that for any $W_g\in \Omega_g,$ we have
    $$
\mathbb{P}_{W_g}(\inf_{y\in S}\|Gy\|_2\leq h_\eqref{propos13.13}\sqrt{N})\leq\exp(-w_\eqref{lemma3.10}N/2).
    $$\end{Proposition}
\begin{proof}
Assume $S\ne\varnothing$. Write
$C:=C_{\eqref{lemma3.444}}$, $C_{\rm net}:=C_{\eqref{lemma3.7}}$,
$h:=h_{\eqref{lemma3.10}}$, and $w:=w_{\eqref{lemma3.10}}$.
By adjusting constants as necessary, we take
$C,C_{\rm net},C_{\eqref{theorem3.2}}\geq1$ and
$h,w,c_{\eqref{lemma3.56}}\leq1$; also $\gamma\leq1/2$.
Set $f_0:=(1-\delta^{-1/4})c_{\eqref{lemma3.56}}\gamma/
C_{\eqref{theorem3.2}}$ and choose
$\tau_0=\tau_0(\gamma,\delta,K)\in(0,1]$ sufficiently small that
$$
B_s:=\frac{64CC_{\rm net}2^{s/2}}{hf_0\tau_0^{3/2}},
\qquad B_s^{\tau_02^{-s/4}}\leq e^{w/4}\quad(s\geq0).
$$
Choose $N_{\eqref{propos13.13}}=N_{\eqref{propos13.13}}(\gamma,\delta,K)$
so that, for $N\geq N_{\eqref{propos13.13}}$,
\begin{equation}\label{firstbound123}
\frac{1}{\lfloor N^{1/4}\rfloor}
\leq\frac{f_0\sqrt{\tau_0}}{16}N^{-3/16},
\qquad \frac{192CC_{\rm net}\sqrt N}{hf_0\tau_0^{3/2}}
\leq e^{wN/4}.
\end{equation}
Lemma \ref{lemma3.56} gives
$\ell\in[0,\lfloor\log_2\sqrt N\rfloor]$, a shift $\lambda$, and sets
$H_1,H_2$ such that $\widehat H_s:=\lambda+H_s$ satisfy
$\operatorname{dist}(\widehat H_1,\widehat H_2)\geq d$ and
$\min_s\mathbb P(g_{ij}\in\widehat H_s)\geq r$, where we set
$$
\begin{gathered}
R:=2^{\ell+2},\quad d:=2^\ell,\quad
r:=c_{\eqref{lemma3.56}}\gamma2^{-\ell/8},\quad
m:=\lceil\tau_0n2^{-\ell/4}\rceil,\quad t:=\frac12\sqrt{m/n},\\
b:=h_{\eqref{lemma3.9}},\qquad \epsilon:=\frac{hb}{2CR}.
\end{gathered}
$$
Here $b$ is evaluated at these $d,r,t$. We have $1\leq m\leq n$ and
$$
\begin{aligned}
b&=\frac{f_0}{8}2^{7\ell/8}\sqrt{m/n}
\geq\frac{f_0\sqrt{\tau_0}}8\,2^{3\ell/4},\\
0<\epsilon&=\frac{hf_0}{64C}2^{-\ell/8}\sqrt{m/n}\leq\frac1{64}.
\end{aligned}
$$
Let $T\subset B_2^n$ consist of the $m$-sparse vectors with
$\|y\|_2\geq t$ and $\|y\|_\infty\leq2b/d$.
The first inequality in \eqref{firstbound123} implies
$1/\lfloor N^{1/4}\rfloor\leq2b/d$, so Lemma \ref{lemma3.82}
gives \eqref{line444}. Lemma \ref{lemma3.7} therefore provides a
nonempty finite net $\mathcal N\subset T$ such that every $y\in S$
admits $y'\in\mathcal N$ with
$\|y\chi_{\operatorname{supp}y'}-y'\|_2\leq\epsilon$, and
$$
|\mathcal N|\leq\left(\frac{C_{\rm net}n}{m\epsilon}\right)^m
\leq B_\ell^{\tau_02^{-\ell/4}n+1}.
$$
The extra exponent $1$ accounts for the ceiling in $m$.
Since $2^{\ell/2}\leq\sqrt N$, our choices of $\tau_0,N$ yield
$$
\log(3|\mathcal N|)
\leq\log(3B_\ell)+\tau_02^{-\ell/4}n\log B_\ell
\leq\frac{wN}{4}+\frac{wn}{4}\leq\frac{wN}{2}.
$$

Apply Lemma \ref{lemma3.10} to the original $g_{ij}$, with sets
$\widehat H_1,\widehat H_2$, for each $y'\in\mathcal N$.
The resulting label sets are independent of $G_0$, and their
intersection $\Omega_g$ satisfies
$\mathbb P(\Psi\notin\Omega_g)\leq|\mathcal N|e^{-wN}\leq e^{-wN/2}$.
For every $W\in\Omega_g$ and every deterministic $G_0$, Lemmas
\ref{lemma3.10} and \ref{lemma3.444} give, outside an event of
$\mathbb P_W$-probability at most
$$
2|\mathcal N|e^{-wN}+e^{-N}
\leq3|\mathcal N|e^{-wN}\leq e^{-wN/2},
$$
both $\|G_1\|\leq C\sqrt N$ and
$\operatorname{dist}(G_1y',V_{G_0,G_1}(E_{y'}))\geq hb\sqrt N$
for every $y'\in\mathcal N$, where
$E_{y'}=\operatorname{span}\{e_j:j\in\operatorname{supp}y'\}$.
Proposition \ref{proposition3.13.13.1} then gives
$$
\inf_{y\in S}\|Gy\|_2\geq(hb-C\epsilon)\sqrt N
\geq\frac{hf_0\sqrt{\tau_0}}{16}\sqrt N
=2h_{\eqref{propos13.13}}\sqrt N,
$$
where $h_{\eqref{propos13.13}}:=hf_0\sqrt{\tau_0}/32>0$.
The strict margin over the stated threshold proves the claim.
\end{proof}

\subsection{Lower bounding singular value: conclusion}
\label{section3.4theends}

We have collected all the necessary materials for the proof of Theorem \ref{universalbacks}.

\begin{proof}[\proofname\ of Theorem \ref{universalbacks}]
Set $a:=\alpha_\eqref{line276definition}$,
$\beta:=\beta_\eqref{line276definition}$, $v:=a^{-1}>0$, and
$\gamma:=\beta/4$. Apply the auxiliary results to
$\widetilde G:=G/a$, whose active entries
$\widetilde g_{ij}:=g_{ij}/a$ are centered, $K/a$-sub-Gaussian, and satisfy
$\mathbb E\widetilde g_{ij}^{\,2}=v^2$ and
$\mathcal Q(\widetilde g_{ij},1)\leq1-\beta$.
Unit variance is not required in those results; all their constants
below are chosen for this rescaled distribution.

Write $X=\widetilde g_{ij}$ and choose $z\in\mathbb R$ such that
$$
\mathbb P(X<z-1)\leq\beta/2\leq\mathbb P(X\leq z-1).
$$
The concentration bound gives $\mathbb P(X>z+1)\geq\beta/2$.
These two tails and $\mathbb EX^2=v^2$ imply
$|z|\leq v\sqrt{2/\beta}-1$. Hence, for $N\geq16v^2/\beta$,
Markov's inequality gives
$$
\max\{\mathbb P(X<z-\sqrt N),\mathbb P(X>z+\sqrt N)\}
\leq\frac{v^2}{(\sqrt N-|z|)^2}\leq\frac\beta4=\gamma.
$$
Subtracting these outer tails yields
$$
\min\{\mathbb P(z-\sqrt N\leq X\leq z-1),
\mathbb P(z+1\leq X\leq z+\sqrt N)\}\geq\gamma,
$$
as required by Propositions \ref{proposition3.11} and \ref{propos13.13}.

Set $w_*:=\min\{w_{\eqref{lemma3.12}},w_{\eqref{lemma3.10}},2\}$ and
choose $N_0=N_0(a,\beta,\delta,K)$ so that
$$
N_0\geq\max\left\{N_{\eqref{proposition3.11}},N_{\eqref{propos13.13}},
\left\lceil16v^2/\beta\right\rceil\right\},
$$
and $N\leq\exp(w_{\eqref{lemma3.12}}N/2)$ and
$4\leq\exp(w_*N/4)$ whenever $N\geq N_0$.
Let $\Omega_p,\Omega_a,\Omega_g$ be the label sets from Lemma
\ref{lemma3.12} and Propositions \ref{proposition3.11}, \ref{propos13.13},
and put $\mathcal D:=\Omega_p\cap\Omega_a\cap\Omega_g$.
These sets work uniformly over the deterministic part $G_0/a$.
A union bound gives
$$
\mathbb P(\Psi\notin\mathcal D)
\leq2e^{-w_{\eqref{lemma3.10}}N/2}+e^{-w_{\eqref{lemma3.12}}N}
\leq3e^{-w_*N/2}\leq e^{-w_*N/4}.
$$
For every $W\in\mathcal D$, the three sphere bounds, with
$h:=\min\{h_{\eqref{proposition3.11}},h_{\eqref{propos13.13}},
h_{\eqref{lemma3.12}}\theta_{\eqref{proposition3.11}}\}$, yield
$$
\begin{aligned}
\mathbb P_W\{\sigma_{min}(\widetilde G)\leq h\sqrt N\}
&\leq ne^{-w_{\eqref{lemma3.12}}N}
+2e^{-w_{\eqref{lemma3.10}}N/2}+e^{-N}\\
&\leq4e^{-w_*N/2}\leq e^{-w_*N/4}.
\end{aligned}
$$
Finally, $\sigma_{min}(G)=a\sigma_{min}(\widetilde G)$, so take
$$
h_{\eqref{universalbacks}}:=a h,\qquad
w_{\eqref{universalbacks}}:=\frac{w_*}{4},\qquad
N_{\eqref{universalbacks}}:=N_0.
$$
The constants $h_{\eqref{universalbacks}}$ and $N_0$ depend only on
$a,\beta,\delta,K$, whereas $w_{\eqref{universalbacks}}$ depends only
on $\beta,\delta$, as asserted.
\end{proof}

\subsection{Proof of universality results}\label{section3.5}

This section contains the proof of Theorem \ref{derivation2.414}. We begin with some notations.

The Hausdorff distance of two subsets $A,B\subset\mathbb{R}$ is defined via 
$$
d_H(A,B):=\inf\{\epsilon>0:A\subset B+[-\epsilon,\epsilon]\text{ and } B\subset A+[-\epsilon,\epsilon]\}.
$$

   For any bounded operator  $X$ on a Hilbert space we use the notation $\operatorname{sp}(X)$ to denote the spectrum of $X$.

We will use the following spectrum universality result from \cite{brailovskaya2024universality}, Theorem 2.6.

The comparison replaces the centered random matrix by a Gaussian matrix with the same covariance structure. Its proof uses a cumulant expansion together with high-order trace and resolvent estimates; the resulting errors are controlled by $\sigma_*(X)$, $\sigma(X)$, and $R(X)$. Concentration then converts these estimates into high-probability Hausdorff control of the spectra. After self-adjoint linearization, this spectral comparison yields a comparison of the smallest singular values.

Let $Z_0$ be an $d\times d$ self-adjoint deterministic matrix, and let $Z_1,\cdots,Z_n$ 
be independent $d\times d$ self-adjoint random matrices with $\mathbb{E}[Z_i]=0$. Consider
\begin{equation}
X:=Z_0+\sum_{i=1}^n Z_i,
\end{equation}
and let $\operatorname{Cov}(X)$ be the $d^2\times d^2$ covariance matrix such that \begin{equation}
    \operatorname{Cov}(X)_{ij,kl}:=\mathbb{E}[(X-\mathbb{E}X)_{ij}\overline{(X-\mathbb{E}X)}_{kl}].
\end{equation}

Let $G$ be a $d\times d$ self-adjoint random matrix where $\{\Re G_{ij},\Im G_{ij}:i,j\in[d]\}$ are jointly Gaussian, satisfying that $\mathbb{E}[G]=\mathbb{E}[X]$ and $\operatorname{Cov}(G)=\operatorname{Cov}(X)$. Then $G$ can be expressed as
$$
G=Z_0+\sum_{i=1}^N A_ig_i
$$
for deterministic matrices $A_1,\cdots,A_N\in M_d(\mathbb{C})_{sa}$ and for i.i.d. real Gaussians $g_1,\cdots,g_N$.

Following \cite{brailovskaya2024universality} we introduce matrix parameters

$$
\sigma(X):=\|\mathbb{E}[(X-\mathbb{E}X)^2]\|^\frac{1}{2},\quad \sigma_*(X):=\sup_{\|v\|_2=\|w\|_2=1}\mathbb{E}[|\langle v,(X-\mathbb{E}X)w\rangle|^2]^\frac{1}{2},
$$ 
$$
R(X):=\left\|\max_{1\leq i\leq n}\|Z_i\|\right\|_\infty.
$$

\begin{theorem}\label{theorem3.14goods}[\cite{brailovskaya2024universality}] There exists a universal constant $C>0$ such that for any $t>0$, 
\begin{equation}
    \mathbb{P}[d_H(\operatorname{sp}(X),\operatorname{sp}(G))> C\epsilon(t)]\leq de^{-t},
\end{equation} where 
$$\epsilon(t)=\sigma_*(X)t^\frac{1}{2}+R(X)^\frac{1}{3}\sigma(X)^\frac{2}{3}t^\frac{2}{3}+R(X)t.
 $$   
\end{theorem}

The following lemma (for deterministic matrices) is similar to \cite{bandeira2023matrix}, Lemma 3.13 and has a similar proof.

\begin{lemma}\label{740new}
Let $N\geq n$ and let $H,G\in\mathbb{R}^{N\times n}$ be deterministic.
For $\epsilon>0$, set $A_\epsilon:=4\epsilon^2 I_n$, where $I_n$ is
the $n$-dimensional identity matrix, and define
    $$
\widetilde{G}_\epsilon=\begin{pmatrix}0&0&G^*&A_\epsilon^\frac{1}{2}\\0&0&0&0\\G&0&0&0\\A_\epsilon^\frac{1}{2}&0&0&0
\end{pmatrix},\quad \widetilde{H}_\epsilon=\begin{pmatrix}0&0&H^*&A_\epsilon^\frac{1}{2}\\0&0&0&0\\H&0&0&0\\A_\epsilon^\frac{1}{2}&0&0&0\end{pmatrix}.
    $$ 
The four block sizes are $n,n,N,n$, so these are self-adjoint matrices
of dimension $N+3n\leq4N$.
    Then $\operatorname{sp}(\widetilde{H}_\epsilon)\subseteq \operatorname{sp}(\widetilde{G}_\epsilon)+[-\epsilon,\epsilon]$
    implies that 
    $$\begin{cases}
\lambda_+(H^*H+A_\epsilon)^\frac{1}{2}\leq \lambda_+(G^*G+A_\epsilon)^\frac{1}{2}+\epsilon,\\
\lambda_{-}(H^*H+A_\epsilon)^\frac{1}{2}\geq \lambda_{-}(G^*G+A_\epsilon)^\frac{1}{2}-\epsilon.\\
    \end{cases}$$
The same implication holds after interchanging $G$ and $H$ in both
the hypothesis and the conclusion. We write
$\lambda_+(Z):=\sup\operatorname{sp}(Z)$ and
$\lambda_{-}(Z):=\inf\operatorname{sp}(Z)$ for self-adjoint $Z$.

    Furthermore, $\operatorname{sp}(\widetilde{
    H}_\epsilon)\subseteq \operatorname{sp}(\widetilde{G}_\epsilon)+[-\epsilon,\epsilon]$ implies that 
    $$
\sigma_{min}(H)\geq\sigma_{min}(G)-3\epsilon.
    $$
\end{lemma}

\begin{proof}
The nonzero eigenvalues of $\widetilde G_\epsilon$ are the positive
and negative singular values of the matrix
$(G^*\ \ A_\epsilon^{1/2})$. Consequently,
$$
\operatorname{sp}(\widetilde{G}_\epsilon)
=\operatorname{sp}((G^*G+A_\epsilon)^\frac{1}{2})
\cup -\operatorname{sp}((G^*G+A_\epsilon)^\frac{1}{2})\cup\{0\},
$$
and the same holds for $\widetilde H_\epsilon$.
The assumed spectral inclusion immediately gives
$$
\lambda_+(H^*H+A_\epsilon)^\frac{1}{2}
\leq\lambda_+(G^*G+A_\epsilon)^\frac{1}{2}+\epsilon.
$$
Every positive eigenvalue of either linearization is at least
$2\epsilon$. In particular, the smallest positive eigenvalue of
$\widetilde H_\epsilon$ cannot lie within distance $\epsilon$ of zero
or of a negative eigenvalue of $\widetilde G_\epsilon$. Its matching
eigenvalue must therefore be positive, which gives
$$
\lambda_{-}(H^*H+A_\epsilon)^\frac{1}{2}
\geq\lambda_{-}(G^*G+A_\epsilon)^\frac{1}{2}-\epsilon.
$$
Finally,
$$
\sigma_{min}(H)
\leq\lambda_{-}(H^*H+A_\epsilon)^\frac{1}{2}
=\sqrt{\sigma_{min}(H)^2+4\epsilon^2}
\leq\sigma_{min}(H)+2\epsilon.
$$
Combining this with the analogous inequalities for $G$ proves
$$
\sigma_{min}(H)\geq \sigma_{min}(G)-3\epsilon.
$$\end{proof}

Let $T$ and $G$ be as in Theorem \ref{derivation2.414}, and form their
linearizations as above. The centered random part has $T_1^*$ in block
$(1,3)$ and $T_1$ in block $(3,1)$; all other blocks are zero.
The two nonzero diagonal blocks of its expected square are
$\mathbb{E}(T_1^*T_1)$ and $\mathbb{E}(T_1T_1^*)$.
Their norms are bounded by $M^2N/n$ and $M^2$, respectively.
Independence of the entries and Cauchy--Schwarz bound the variance of
every unit-vector bilinear form by $M^2/n$. Moreover, each
self-adjoint increment corresponding to an entry $t_{ij}$ has norm
$|t_{ij}|\leq q$. Thus, for every $\epsilon>0$,
$$
\sigma(\widetilde{T}_\epsilon)\leq M\sqrt{\frac{N}{n}},\quad \sigma_*(\widetilde{T}_\epsilon)\leq\frac{M}{\sqrt{n}}, \quad R(\widetilde{T}_\epsilon)\leq q.
$$

We have collected necessary ingredients for proving Theorem \ref{derivation2.414}.
\begin{proof}[\proofname\ of Theorem \ref{derivation2.414}]
We define random matrices $\widetilde{T}_{\epsilon(t)},$ $\widetilde{G}_{\epsilon(t)}$ as in the statement of Lemma \eqref{740new}, where we replace $\epsilon$ by $\epsilon(t)$ defined by 
$$
\epsilon(t)=CMn^{-\frac{1}{2}}t^\frac{1}{2}+CM^\frac{2}{3}q^{\frac{1}{3}}t^{\frac{2}{3}}(\frac{N}{n})^\frac{1}{3}+Cqt,
$$ where the constant $C>0$ is defined in Theorem \ref{theorem3.14goods} and $M>1$ is the upper bound of entry variance in the statement of Theorem \ref{derivation2.414}.
Applying Theorem \ref{theorem3.14goods} in dimension $N+3n\leq4N$, we get that for any $t>0$,
$$
\mathbb{P}\left(d_H\left(\operatorname{sp}(\widetilde{T}_{\epsilon(t)}),\operatorname{sp}(\widetilde{G}_{\epsilon(t)})\right)> \epsilon(t)\right)\leq 4N e^{-t}.
    $$
With probability at least $1-4Ne^{-t}$, Lemma \ref{740new}, applied
in both directions, therefore gives
$$
|\sigma_{min}(T)-\sigma_{min}(G)|\leq3\epsilon(t).
$$
Since $M>1$ and $q,t>0$, the threshold in Theorem
\ref{derivation2.414} with constant $3CM$ is strictly larger than
$3\epsilon(t)$. This proves the claimed bound.
\end{proof}

\section*{Funding}
The author is supported by a Simons Foundation Grant (601948, DJ)

\printbibliography

@book{bingham1987regular,
  title={Regular Variation},
  author={Bingham, Nicholas H. and Goldie, Charles M. and Teugels, Jozef L.},
  series={Encyclopedia of Mathematics and its Applications},
  volume={27},
  publisher={Cambridge University Press},
  year={1987}
}

@article{tikhomirov2015limit,
  title={The limit of the smallest singular value of random matrices with iid entries},
  author={Tikhomirov, Konstantin},
  journal={Advances in Mathematics},
  volume={284},
  pages={1--20},
  year={2015},
  publisher={Elsevier}
}

@inproceedings{rudelson2010non,
  title={Non-asymptotic theory of random matrices: extreme singular values},
  author={Rudelson, Mark and Vershynin, Roman},
  booktitle={Proceedings of the International Congress of Mathematicians 2010 (ICM 2010) (In 4 Volumes) Vol. I: Plenary Lectures and Ceremonies Vols. II--IV: Invited Lectures},
  pages={1576--1602},
  year={2010},
  organization={World Scientific}
}

@article{dumitriu2024extreme,
  title={Extreme singular values of inhomogeneous sparse random rectangular matrices},
  author={Dumitriu, Ioana and Zhu, Yizhe},
  journal={Bernoulli},
  volume={30},
  number={4},
  pages={2904--2931},
  year={2024},
  publisher={Bernoulli Society for Mathematical Statistics and Probability}
}

@article{wojtaszczyk2010stability,
  title={Stability and instance optimality for Gaussian measurements in compressed sensing},
  author={Wojtaszczyk, Przemyslaw},
  journal={Foundations of Computational Mathematics},
  volume={10},
  pages={1--13},
  year={2010},
  publisher={Springer}
}

@article{mankiewicz2000compact,
  title={Compact Groups of Operators on Subproportional Quotients of l1m},
  author={Mankiewicz, Piotr},
  journal={Canadian Journal of Mathematics},
  volume={52},
  number={5},
  pages={999--1017},
  year={2000},
  publisher={Cambridge University Press}
}

@article{litvak2002randomized,
  title={Randomized isomorphic Dvoretzky theorem},
  author={Litvak, Alexander and Mankiewicz, Piotr and Tomczak-Jaegermann, Nicole},
  journal={Comptes Rendus Mathematique},
  volume={335},
  number={4},
  pages={345--350},
  year={2002},
  publisher={Elsevier}
}

@article{bao2024phase2,
  title={Phase transition for the bottom singular vector of rectangular random matrices},
  author={Bao, Zhigang and Lee, Jaehun and Xu, Xiaocong},
  journal={arXiv preprint arXiv:2409.01819},
  year={2024}
}

@article{tao2010random,
  title={Random matrices: The distribution of the smallest singular values},
  author={Tao, Terence and Vu, Van},
  journal={Geometric And Functional Analysis},
  volume={20},
  pages={260--297},
  year={2010},
  publisher={Springer}
}

@article {MR1235416,
    AUTHOR = {Bai, Z. D. and Yin, Y. Q.},
     TITLE = {Limit of the smallest eigenvalue of a large-dimensional sample
              covariance matrix},
   JOURNAL = {Ann. Probab.},
  FJOURNAL = {The Annals of Probability},
    VOLUME = {21},
      YEAR = {1993},
    NUMBER = {3},
     PAGES = {1275--1294},
      ISSN = {0091-1798,2168-894X},
   MRCLASS = {60F15 (62H99)},
  MRNUMBER = {1235416},
}

@article {louvaris2022universality,
    AUTHOR = {Louvaris, Michail},
     TITLE = {Universality of the least singular value and singular vector
              delocalisation for {L}\'evy non-symmetric random matrices},
   JOURNAL = {Ann. Inst. Henri Poincar\'e{} Probab. Stat.},
  FJOURNAL = {Annales de l'Institut Henri Poincar\'e{} Probabilit\'es et
              Statistiques},
    VOLUME = {61},
      YEAR = {2025},
    NUMBER = {1},
     PAGES = {30--82},
      ISSN = {0246-0203,1778-7017},
   MRCLASS = {60B20 (15B52 60E07 60F15)},
}

@article{tikhomirov2016smallest,
  title={The smallest singular value of random rectangular matrices with no moment assumptions on entries},
  author={Tikhomirov, Konstantin E},
  journal={Israel Journal of Mathematics},
  volume={212},
  pages={289--314},
  year={2016},
  publisher={Springer}
}

@article{bandeira2023matrix,
  title={Matrix concentration inequalities and free probability},
  author={Bandeira, Afonso S and Boedihardjo, March T and van Handel, Ramon},
  journal={Inventiones mathematicae},
  volume={234},
  number={1},
  pages={419--487},
  year={2023},
  publisher={Springer}
}

@article {livshyts2021smallest,
    AUTHOR = {Livshyts, Galyna V. and Tikhomirov, Konstantin and Vershynin,
              Roman},
     TITLE = {The smallest singular value of inhomogeneous square random
              matrices},
   JOURNAL = {Ann. Probab.},
  FJOURNAL = {The Annals of Probability},
    VOLUME = {49},
      YEAR = {2021},
    NUMBER = {3},
     PAGES = {1286--1309},
      ISSN = {0091-1798,2168-894X},
   MRCLASS = {60B20 (15B52)},
  MRNUMBER = {4255145},

}

@article {bao2024phase,
    AUTHOR = {Bao, Zhigang and Lee, Jaehun and Xu, Xiaocong},
     TITLE = {Phase transition for the smallest eigenvalue of covariance
              matrices},
   JOURNAL = {Probab. Theory Related Fields},
  FJOURNAL = {Probability Theory and Related Fields},
    VOLUME = {192},
      YEAR = {2025},
    NUMBER = {1-2},
     PAGES = {35--133},
      ISSN = {0178-8051,1432-2064},
}

@article{han2024deformed,
  title={Deformed Fr{\'e}chet law for Wigner and sample covariance matrices with tail in crossover regime},
  author={Han, Yi},
  journal={Probability Theory and Related Fields},
  pages={1--28},
  year={2024},
  publisher={Springer}
}

@article {MR4260468,
    AUTHOR = {Aggarwal, Amol and Lopatto, Patrick and Marcinek, Jake},
     TITLE = {Eigenvector statistics of {L}\'evy matrices},
   JOURNAL = {Ann. Probab.},
  FJOURNAL = {The Annals of Probability},
    VOLUME = {49},
      YEAR = {2021},
    NUMBER = {4},
     PAGES = {1778--1846},
      ISSN = {0091-1798,2168-894X},
   MRCLASS = {60B20 (15B52)},
  MRNUMBER = {4260468},
MRREVIEWER = {Deli\ Li},
}

@article {MR4310816,
    AUTHOR = {Aggarwal, Amol and Lopatto, Patrick and Yau, Horng-Tzer},
     TITLE = {G{OE} statistics for {L}\'evy matrices},
   JOURNAL = {J. Eur. Math. Soc. (JEMS)},
  FJOURNAL = {Journal of the European Mathematical Society (JEMS)},
    VOLUME = {23},
      YEAR = {2021},
    NUMBER = {11},
     PAGES = {3707--3800},
      ISSN = {1435-9855,1435-9863},
   MRCLASS = {60B20},
  MRNUMBER = {4310816},
}

@article{aggarwal2019bulk,
  title={Bulk universality for generalized Wigner matrices with few moments},
  author={Aggarwal, Amol},
  journal={Probability Theory and Related Fields},
  volume={173},
  pages={375--432},
  year={2019},
  publisher={Springer}
}

@article{lehner1999computing,
  title={Computing norms of free operators with matrix coefficients},
  author={Lehner, Franz},
  journal={American Journal of Mathematics},
  volume={121},
  number={3},
  pages={453--486},
  year={1999},
  publisher={Johns Hopkins University Press}
}

@article{rudelson2009smallest,
  title={Smallest singular value of a random rectangular matrix},
  author={Rudelson, Mark and Vershynin, Roman},
  journal={Communications on Pure and Applied Mathematics: A Journal Issued by the Courant Institute of Mathematical Sciences},
  volume={62},
  number={12},
  pages={1707--1739},
  year={2009},
  publisher={Wiley Online Library}
}

@article{litvak2005smallest,
  title={Smallest singular value of random matrices and geometry of random polytopes},
  author={Litvak, Alexander E and Pajor, Alain and Rudelson, Mark and Tomczak-Jaegermann, Nicole},
  journal={Advances in Mathematics},
  volume={195},
  number={2},
  pages={491--523},
  year={2005},
  publisher={Elsevier}
}

@incollection {MR4680362,
    AUTHOR = {Tikhomirov, Konstantin},
     TITLE = {Quantitative invertibility of non-{H}ermitian random matrices},
 BOOKTITLE = {I{CM}---{I}nternational {C}ongress of {M}athematicians. {V}ol.
              4. {S}ections 5--8},
     PAGES = {3292--3313},
 PUBLISHER = {EMS Press, Berlin},
      YEAR = {2023},
     
}

@article{basak2017invertibility,
  title={Invertibility of sparse non-Hermitian matrices},
  author={Basak, Anirban and Rudelson, Mark},
  journal={Advances in Mathematics},
  volume={310},
  pages={426--483},
  year={2017},
  publisher={Elsevier}
}

@article{benhamou2018operator,
  title={Operator norm upper bound for sub-Gaussian tailed random matrices},
  author={Benhamou, Eric and Atif, Jamal and Laraki, Rida},
  journal={arXiv preprint arXiv:1812.09618},
  year={2018}
}

@article{gotze2023largest,
  title={On the largest and the smallest singular value of sparse rectangular random matrices},
  author={G{\"o}tze, Friedrich and Tikhomirov, A},
  journal={Electronic Journal of Probability},
  volume={28},
  pages={1--18},
  year={2023},
  publisher={The Institute of Mathematical Statistics and the Bernoulli Society}
}

@article{guedon2020random,
  title={Random polytopes obtained by matrices with heavy-tailed entries},
  author={Gu{\'e}don, Olivier and Litvak, Alexander E and Tatarko, K},
  journal={Communications in Contemporary Mathematics},
  volume={22},
  number={04},
  pages={1950027},
  year={2020},
  publisher={World Scientific}
}

@article{donoho2006compressed,
  title={Compressed sensing},
  author={Donoho, David L},
  journal={IEEE Transactions on information theory},
  volume={52},
  number={4},
  pages={1289--1306},
  year={2006},
  publisher={IEEE}
}

@article{candes2006stable,
  title={Stable signal recovery from incomplete and inaccurate measurements},
  author={Candes, Emmanuel J and Romberg, Justin K and Tao, Terence},
  journal={Communications on Pure and Applied Mathematics: A Journal Issued by the Courant Institute of Mathematical Sciences},
  volume={59},
  number={8},
  pages={1207--1223},
  year={2006},
  publisher={Wiley Online Library}
}

@inproceedings{candes2006compressive,
  title={Compressive sampling},
  author={Cand{\`e}s, Emmanuel J and others},
  booktitle={Proceedings of the international congress of mathematicians},
  volume={3},
  pages={1433--1452},
  year={2006},
  organization={Madrid, Spain}
}

@article{hayakawa2023estimating,
  title={Estimating the probability that a given vector is in the convex hull of a random sample},
  author={Hayakawa, Satoshi and Lyons, Terry and Oberhauser, Harald},
  journal={Probability Theory and Related Fields},
  volume={185},
  number={3},
  pages={705--746},
  year={2023},
  publisher={Springer}
}

@article{guedon2022geometry,
  title={On the geometry of polytopes generated by heavy-tailed random vectors},
  author={Gu{\'e}don, Olivier and Krahmer, Felix and K{\"u}mmerle, Christian and Mendelson, Shahar and Rauhut, Holger},
  journal={Communications in Contemporary Mathematics},
  volume={24},
  number={03},
  pages={2150056},
  year={2022},
  publisher={World Scientific}
}

@article{marchenko1967distribution,
  title={Distribution of eigenvalues for some sets of random matrices},
  author={Marchenko, Vladimir Alexandrovich and Pastur, Leonid Andreevich},
  journal={Matematicheskii Sbornik},
  volume={114},
  number={4},
  pages={507--536},
  year={1967},
  publisher={Russian Academy of Sciences, Steklov Mathematical Institute of Russian~…}
}

@article{soshnikov2002note,
  title={A note on universality of the distribution of the largest eigenvalues in certain sample covariance matrices},
  author={Soshnikov, Alexander},
  journal={Journal of Statistical Physics},
  volume={108},
  pages={1033--1056},
  year={2002},
  publisher={Springer}
}

@article{feldheim2010universality,
  title={A universality result for the smallest eigenvalues of certain sample covariance matrices},
  author={Feldheim, Ohad N and Sodin, Sasha},
  journal={Geometric And Functional Analysis},
  volume={20},
  number={1},
  pages={88--123},
  year={2010},
  publisher={Springer}
}

@article{belinschi2009spectral,
  title={Spectral measure of heavy tailed band and covariance random matrices},
  author={Belinschi, Serban and Dembo, Amir and Guionnet, Alice},
  journal={Communications in Mathematical Physics},
  volume={289},
  number={3},
  pages={1023--1055},
  year={2009},
  publisher={Springer}
}

@inproceedings{auffinger2009poisson,
  title={Poisson convergence for the largest eigenvalues of heavy tailed random matrices},
  author={Auffinger, Antonio and Ben Arous, G{\'e}rard and P{\'e}ch{\'e}, Sandrine},
  booktitle={Annales de l'IHP Probabilit{\'e}s et statistiques},
  volume={45},
  number={3},
  pages={589--610},
  year={2009}
}

@article {MR131894,
    AUTHOR = {Rogozin, B. A.},
     TITLE = {On the increase of dispersion of sums of independent random
              variables},
   JOURNAL = {Teor. Verojatnost. i Primenen.},
  FJOURNAL = {Akademija Nauk SSSR. Teorija Verojatnoste\u i\ i ee
              Primenenija},
    VOLUME = {6},
      YEAR = {1961},
     PAGES = {106--108},
      ISSN = {0040-361x},
   MRCLASS = {60.30},
  MRNUMBER = {131894},
MRREVIEWER = {G.\ Maruyama},
}

@article {brailovskaya2024universality,
    AUTHOR = {Brailovskaya, Tatiana and van Handel, Ramon},
     TITLE = {Universality and sharp matrix concentration inequalities},
   JOURNAL = {Geom. Funct. Anal.},
  FJOURNAL = {Geometric and Functional Analysis},
    VOLUME = {34},
      YEAR = {2024},
    NUMBER = {6},
     PAGES = {1734--1838},
      ISSN = {1016-443X,1420-8970},
   MRCLASS = {60B20 (15B52 46L53 46L54 60E15)},
}

@book{rockafellar1970convex,
  author = {Rockafellar, R. Tyrrell},
  title = {Convex Analysis},
  series = {Princeton Mathematical Series},
  volume = {28},
  publisher = {Princeton University Press},
  address = {Princeton, NJ},
  year = {1970},
}

\end{document}